%% file: pairings.tex
\documentclass[a4paper,10pt,final]{article}
\usepackage{ifthen}

\newboolean{PSdiagram}

\setboolean{PSdiagram}{true} %

\input{preamble}
\title{Cup Products and Pairings for Abelian Varieties}
\author{Klaus Loerke}
\date{2009}

\begin{document}

\maketitle

\begin{abstract}
Let $A_K$ be an abelian variety with semistable reduction over a strictly henselian field 
of positive characteristic with perfect residue class field. We show that there is a close 
connection between the pairings of Grothendieck, Bester/Bertapelle and Shafarevic. In particular,
we show that the pairing of Bester/Bertapelle can be used to describe the $p$-part of 
Grothendieck's pairing in the semistable reduction case, thus proving a conjecture of Bertapelle,
\cite{be03}.
\end{abstract}

\input{intro}

\input{cupder}
\input{bester}

\input{newbester}

\input{bib}

\end{document}

%% file: preamble.tex
\usepackage[english]{babel}
\usepackage[latin1]{inputenc}
\usepackage{mathrsfs}
\usepackage{amssymb}
\usepackage{amsmath}
\usepackage{dsfont}
\usepackage[heavycircles]{stmaryrd}
\usepackage{mathtools}

\usepackage{theorem}
\ifthenelse{\boolean{PSdiagram}} %
{
  \usepackage{diagrams}
  \diagramstyle[height=0.20in,width=0.30in,labelstyle=\scriptstyle,dpi=600,nohug,PostScript=dvips]
}{
  \usepackage[UglyObsolete]{diagrams}
  \diagramstyle[height=0.20in,width=0.30in,labelstyle=\scriptstyle,dpi=600,nohug,noPostScript]
}

\renewcommand\textrm[1]{{\rm #1}}

\newlength{\listindent}
\newcommand\enumroman[1]{{\rm(\textit{\roman{#1}})}}
\newcommand\enumarabic[1]{\arabic{#1}.}
\newcounter{enumcounter}
\newcommand\enum[1]{\setcounter{enumcounter}{#1}\enumroman{enumcounter}}

\newenvironment{items}
   {\begin{list}{\textbullet} {
                    \setlength{\topsep}{0.1ex plus0.1ex minus0.1ex}
                    \setlength{\itemsep}{0.1ex plus0.1ex minus0.1ex}
                    \setlength{\parsep}{0.1ex plus0.1ex minus0.1ex}
                    \setlength{\itemindent}{0in}
                    \setlength{\labelwidth}{1ex}
                    \setlength{\leftmargin}{\listindent}
                    \setlength{\partopsep}{\medskipamount} } }
   {\end{list}}

\newenvironment{enumprop}
   {\begin{list}{\enumroman{enumiv}}{\usecounter{enumiv}
                      \setlength{\topsep}{0.1ex plus0.1ex minus0.1ex}
                      \setlength{\itemsep}{0.1ex plus0.1ex minus0.1ex}
                      \setlength{\parsep}{0.1ex plus0.1ex minus0.1ex}
                      \setlength{\itemindent}{0in}
                      \setlength{\labelwidth}{2.5em}
                      \setlength{\leftmargin}{2\listindent}
                      \setlength{\partopsep}{\medskipamount} } }
   {\end{list}}

\newenvironment{enumsrm}
   {\begin{list}{\enumroman{enumiv}}{\usecounter{enumiv}
                      \setlength{\topsep}{0.1ex plus0.1ex minus0.1ex}
                      \setlength{\itemsep}{0.1ex plus0.1ex minus0.1ex}
                      \setlength{\parsep}{0.1ex plus0.1ex minus0.1ex}
                      \setlength{\itemindent}{0in}
                      \setlength{\labelwidth}{2.5em}
                      \setlength{\leftmargin}{2\listindent}
                      \setlength{\partopsep}{\medskipamount} } }
   {\end{list}}

\newcommand\varhspace[1][1em]{\hspace{0em plus#1 minus0.3em}}

\newcommand\proofpart[1]{\textit{#1}}
\newcommand\statement[1]{\proofpart{Statement \enum{#1}}}

\newarrowtail{<=}\Leftarrow\Rightarrow\Uparrow\Downarrow

\newarrow{Corresponds} <--->
\newarrow{Equal} =====
\newarrow{Equiv} {<=}==={=>}
\newarrow{Dash} {}{dash}{}{dash}{}
\newarrow{Empty} {}{}{}{}{}
\newarrow{Dashinto} C{dash}{}{dash}>
\newarrow{Dashonto} {}{dash}{}{dash}{>>}
\newarrow{Dotmapsto} {mapsto}...>

\newarrow{CTo} ----{->}
\newarrow{CMapsto} |---{->}
\newarrow{CInto} C---{->}
\newarrow{COnto} ----{->>}

\def\nepok{\overprint{\raise1.5ex\rlap{$\;\;\urcorner$}}}
\def\nwpok{\overprint{\raise2.5ex\llap{$\lrcorner\;\;$}}}
\def\sepbk{\overprint{\lower1.5ex\rlap{$\;\;\lrcorner$}}}

\newcommand\diagramiso{\hbox{\lower1.5ex\hbox{$_\sim$}}}

\theoremstyle{change} %
\theorembodyfont{\it}

\newtheorem{defn}{Definition.}[subsection]

\newtheorem{lemma}[defn]{Lemma.}
\newtheorem{prop}[defn]{Proposition.}
\newtheorem{thm}[defn]{Theorem.}
\newtheorem{cor}[defn]{Corollary.}

\newenvironment{interlude}{\begin{trivlist}\setlength\topsep{\theorempreskipamount}\item}{\end{trivlist}}

\newenvironment{remark}{\begin{interlude}\noindent\textbf{Remark.\ }\nopagebreak}{\end{interlude}}
\newenvironment{remarks}{\begin{interlude}\noindent\textbf{Remarks.\ }\nopagebreak\begin{enumsrm}\nopagebreak}{\end{enumsrm}\end{interlude}}
\newenvironment{remarks*}{\begin{interlude}\noindent\textbf{Remarks.\ }\nopagebreak}{\end{interlude}}

\newenvironment{examples*}{\begin{interlude}\noindent\textbf{Examples.\ }\nopagebreak}{\end{interlude}}
\newenvironment{notation}{\begin{interlude}\noindent\textbf{Notation.\ }\nopagebreak}{\end{interlude}}

\newenvironment{proof}{\begin{interlude}\noindent\textit{Proof.}}{\qed\end{interlude}}
\newenvironment{proofof}[1]{\begin{interlude}\noindent\textit{Proof of}\ \textit{#1}.}{\qed\end{interlude}}
\newenvironment{proof*}{\begin{interlude}\noindent\noindent\textit{Proof.}}{\end{interlude}}

\newcommand\qedsymbol{$\square$}
\newcommand\qed{\nopagebreak\penalty10000\hbox{}\nobreak\penalty10000\nopagebreak\penalty10000\hfill\hbox{\qedsymbol}\par}

\newcommand{\Cech}{\v{C}ech}

\def\oldphi{\mathchar"11E}
\def\phi{\varphi}
\def\theta{\vartheta}
\def\epsilon{\varepsilon}
\def\sub{\subseteq}

\def\mathbb{\mathds}

\newcommand\bbN{{\mathbb{N}}}
\newcommand\bbZ{{\mathbb{Z}}}
\newcommand\bbQ{{\mathbb{Q}}}

\newcommand\bbG{{\mathbb{G}}}

\newcommand\bbQZ{\bbQ\mod\bbZ}

\DeclareMathOperator{\id}{id} %
\DeclareMathOperator{\dlog}{dlog} %
\DeclareMathOperator{\Lie}{Lie} %
\DeclareMathOperator{\coker}{coker} %
\DeclareMathOperator{\preim}{pre\kern0.13em im} %
\DeclareMathOperator{\preker}{pre\kern0.13em ker} %
\DeclareMathOperator{\precoker}{pre\kern0.13em coker} %
\renewcommand{\hom}{\operatorname{Hom}}
\DeclareMathOperator{\tot}{Tot} %
\DeclareMathOperator{\quot}{quot} %
\DeclareMathOperator{\spec}{Spec} %
\DeclareMathOperator{\res}{res} %

\newcommand\et{\textrm{\acute{e}t}}

\newcommand\fl{\textrm{f\/l}}

\newcommand\gp{\textrm{\rm GP}}
\newcommand\bp{\textrm{\rm BP}}
\newcommand\SP{\textrm{\rm SP}}

\def\iso{\cong}

\newcommand\frakm{\mathfrak m}

\newcommand\defeq{\coloneqq}
\newcommand\fedeq{\eqqcolon}
\newcommand\then{\Rightarrow}

\newcommand\set[2][auto]{
     \ifthenelse{\equal{#1}{auto}}{\left\lbrace}{\csname #1\endcsname\lbrace} #2 \ifthenelse{\equal{#1}{auto}}{\right\rbrace}{\csname #1\endcsname\rbrace} }

\renewcommand\mod{/}

\let\tensor\otimes

\newcommand{\indlim}{\operatorname*{\underrightarrow{\lim}}} %

\newcommand\longto{\longrightarrow}
\newcommand\isoto{\longto\!\!\!\!\!\!\!\!^\sim\,\,\,}
\newcommand\isoot{\mbox{$\hspace{8.5pt}\raise 3pt\hbox{$\sim$}\hspace{-18pt}\longleftarrow\hspace{3pt}$}\linebreak[0]}
\newcommand\into{\hookrightarrow}

\newcommand\longonto{\longto\hspace{-9.5pt}\to}

\newcommand\Longto[1]{\stackrel {#1}\longto}
\newcommand\Longot[1]{\stackrel {#1}\longleftarrow}

\newcommand\Longonto[1]{\stackrel {#1}\longonto}

\newcommand\calC{\mathcal C}

\newcommand\ext{\operatorname{Ext}}
\newcommand\tor{\operatorname{Tor}}

\newcommand{\HH}{{\rm H}} %
\newcommand{\RR}{{\rm R}} %
\newcommand{\LL}{{\rm L}} %
\newcommand{\KK}{{\rm K}} %

\newcommand{\cHH}{\check{\rm H}} %
\newcommand{\cGamma}{\check{\Gamma}} %

\newcommand\ssto{\Longrightarrow} 

\DeclareMathOperator{\DD}{D}

\newcommand\scrA{\mathscr A}
\newcommand\scrB{\mathscr B}

\newcommand\scrF{\mathscr F}
\newcommand\scrG{\mathscr G}
\newcommand\scrH{\mathscr H}

\newcommand\scrO{\mathscr O}
\newcommand\scrP{\mathscr P}

\newcommand\scrS{\mathscr S}

\DeclareMathOperator\shom{{\scrH\!\!\!\it o\!m}} %

\DeclareMathOperator\rhom{RHom}
\DeclareMathOperator\rshom{R\!{\scrH\!\!\!\it o\!m}}
\newcommand{\tensorl}{\tensor^\LL}

\def\eps{\varepsilon}

\newcommand\frakU{\mathfrak U}  %

%% file: intro.tex
\section{Introduction}\label{gp_intro}

Let $R$ be a discrete valuation ring with field of fractions $K$ and residue
class field $k$. We consider an abelian variety $A_K$ over $K$ with dual abelian
variety $A'_K$ and N\'eron models $A$ and $A'$, respectively. In \cite{sga7},
exp.\ IX, 1.2.1, Grothendieck constructed a canonical pairing
$\oldphi\times\oldphi'\to\bbQZ$ of groups of components $\oldphi$ of $A$ and
$\oldphi'$ of $A'$. \emph{Grothendieck's pairing} represents the obstruction to
extend the Poincar\'e bundle on $A_K\times_K A_K'$ to a Poincar\'e bundle on
$A\times_R A'$. It is conjectured that this pairing is perfect.
\label{gp_known_facts} Indeed, it can be shown that this pairing is perfect in
the following cases: 
$A_K$ has semistable reduction, \cite{We}, $R$ is of equal characteristic $0$,
\cite{sga7}, $R$ has mixed characteristic and a perfect residue class field, \cite{beg}, 
the residue class field $k$ is finite, \cite{mcc86}.
\par
If $k$ is not perfect, counterexamples can be found, \cite{bb}, corollary 2.5.
If $R$ is of equal characteristic $p\neq 0$ with infinite, perfect residue class
field, there are only partial results: If $A_K$ is of potentially multiplicative
reduction, then Grothendieck's pairing is perfect, \cite{Bo97}.  
\par
In this article, we focus on the open case of discrete valuations rings of equal
characteristic $p\neq 0$ with perfect residue class field $k$. Since the
formation of N\'eron models is compatible with unramified base change, without
loss of generality, we may assume that $R$ is complete and strictly henselian.
\par
In this setting, we enlighten the relation of Grothendieck's pairing to the
pairings of Bester and Shafarevic: The latter pairing 
\[
\SP\colon \HH^1(K,A_K)_{(p)}\times\pi_1(A')_{(p)}\longto\bbQZ
\]
is constructed in \cite{sh} for the prime-to-$p$-parts. For the
prime-to-$p$-parts, the relation of the pairings of Grothendieck and Shafarevic
is known and can be summarised in the following exact sequence of pairings:
\[
 0\longto 
\left(\begin{matrix}
        \text{Grothendieck's}\\ \text{pairing}
\end{matrix}\right)
\longto
\left(\begin{matrix}
        \text{cup product for}\\ \text{$A_{K,n}$ and $A'_{K,n}$}
\end{matrix}\right)
\longto
\left(\begin{matrix}
        \text{Shafarevic's}\\ \text{pairing}
\end{matrix}\right)
\longto 0.
\]
More precisely, there is the following commutative diagram with exact lines. The
inductive limits vary over all $n$ which are prime to $p$:
\begin{diagram}[width=.45in,tight]
0 & \rTo & (\oldphi_A)_{(p)}        & & \rTo & \indlim \HH^1(K,A_{K,n}) 
  & &\rTo & 
\HH^1(K,A_K)_{(p)} & \rTo & 0\phantom. \\
  &      & \dTo>\gp                 & &      & \dTo>\cup                               & &     & 
\dTo>\SP     \\
0 & \rTo & (\oldphi_{A'})_{(p)}^\ast & & \rTo & \indlim \HH^0(K,A'_{K,n})^\ast & & \rTo &
\pi_1(A')_{(p)}^\ast & \rTo & 0. %
\end{diagram}
The vertical morphisms are induced by Grothendieck's pairing, by the perfect cup
product (\cite{sga5}, I, 5.1.8) and by Shafarevic's pairing (cf.\ \cite{be03},
4.1, prop.\ 2 and theorem 2 together with remark 1 on p.\ 152).
\
\\
\par
The $p$-parts of the pairings of Grothendieck and Shafarevic are rather mysterious:
In order to extend Shafarevic's pairing to the $p$-parts  of $\HH^1(K,A_K)$ and $\pi_1(A')$ in the
case of good reduction, Bester constructed a pairing 
\[
 \HH^2_k(R,A_{p^n})\times\scrF(A'_{p^n})\longto\bbQZ
\]
of the kernels of $p^n$-multiplication of the N\'eron models of an abelian variety of good
reduction and its dual, cf.\ \cite{bes}. Here, $\HH^\ast_k$ denotes the functor of local cohomology
and $\scrF$ denotes a rather complicated functor constructed by Bester.
\par
In \cite{be03}, Bertapelle generalised Bester's pairing to the kernels of
$p^n$-multi\-plica\-tion of abelian varieties with semistable reduction. Using
this, she extended Shafarevic's pairing to the $p$-parts in the case of
semistable reduction.
\par
Surprisingly, if we replace the cup product in the above diagram by Bester's
pairing, there is a similar exact sequence of the $p$-parts of these pairings:
(we still have to assume that $A_K$ has semistable reduction)
\begin{equation}\label{pairing_seq}
0\longto 
\left(\begin{matrix}
        \text{pairing of $p$-parts of}\\ \text{groups of components}
\end{matrix}\right)
\longto
\left(\begin{matrix}
        \text{Bester's}\\ \text{pairing}
\end{matrix}\right)
\longto
\left(\begin{matrix}
        \text{Shafarevic's}\\ \text{pairing}
\end{matrix}\right)
\longto 0
\end{equation} 
This sequence, of course, has to be read like the above sequence of pairings. 
\par
The first pairing of this sequence is a perfect pairing of the $p$-parts of
groups of components of $A$ and $A'$. Bertapelle conjectured that this pairing
coincides with Grothendieck's pairing. We will prove this conjecture.
\par
In a final step, we show that the pairing of Shafarevic can be seen as a
higher-dimensional analogue of the pairing of Grothendieck. To make this
explicit, we generalise Bester's functor $\scrF$ to some kind of ``derived
functor''. In this view, the pairings of Grothendieck, Bester and Shafarevic
induce a duality of the homotopy and local cohomology of $A$. 
\
\\
\par
In order to prove the conjecture of A.\ Bertapelle, we need to gain a deep
understanding of Bester's pairing. Bester uses cup products in the derived
category to construct his pairing. Since there are no references to (derived)
cup products on arbitrary sites in the literature, we start by developing such a
theory in greater generality as for example in \cite{ec}, V, \S1, prop.\ 1.16.
\par
In the third section we use the theory of derived cup products to analyse
Bester's pairing and compare it to the pairing of Grothendieck. For this result
we use the technique of rigid uniformisation of abelian varieties: The duality
of $A_K$ and $A_K'$ induces a \emph{monodromy} pairing, which -- on the one hand
-- is compatible with the pairing of Grothendieck and -- on the other hand --
allows for a connection to cup products and thus, for a connection to the
pairing of Bester.

%% file: cupder.tex
\section{Cup Products in the Derived Category}\label{cup_derived}

\subsection[Cup-Functors]{$\cup$-Functors}

Let $\scrS$ be a category of sheaves, e.\,g.\ the category of abelian sheaves on
a Grothen\-dieck topology. Under some hypothesis (cf.\ \cite{ec}, V, \S1, prop
1.16) one can extend a functorial pairing of left exact functors $f_\ast A\times
g_\ast B\to h_\ast(A\tensor B)$ to a \emph{cup product}
\[
\RR^rf_\ast A\times \RR^s g_\ast B\to \RR^{r+s}h_\ast (A\tensor B)
\]
of the derived functors for every $r,s\in\bbN$ which is compatible with the
connecting morphisms $d$ and such that for $r=s=0$ this is the pairing we
started with. It is known how to construct such cup products in the case of
\Cech\ cohomology using \Cech\ cocycles (\cite{ec}, V, \S 1, remark 1.19) and in
the case of group cohomology (\cite{Br}, V.3) using the Bar resolution. To
clarify these various cup products and their relations, we will develop a theory
of \emph{derived cup products} in arbitrary abelian categories $\scrA$. \par
When working with derived categories, we will mainly work in the derived
category of complexes that are bounded from below, $\DD^+\!\scrA$, that is, we
require that for every complex $X\in\DD^+\!\scrA$ there exists $n\in\bbZ$ such
that for all $i<n$ we have $\HH^i(X)=0$.
\par
In general, a left exact functor $f_\ast\colon\scrA\to\scrB$ does not extend
without more ado to the derived category. To remedy this defect, one constructs
the derived functor $\RR f_\ast$. This functor turns distinguished triangles
into distinguished triangles. Moreover, it is equipped with a functorial,
universal morphism $Q\circ f_\ast \to\RR f_\ast\circ Q$, for the canonical
quotient functor $Q\colon\KK^+\!\scrA\to\DD^+\!\scrA$, from the category of
complexes up to homotopy, which are bounded from below, to the derived category.
\par
Composition of functors has a simple description in the setting of derived
categories. If $f_\ast$ and $h_\ast$ are functors whose composition exists and
if $f_\ast$ turns injective objects into $h_\ast$-acyclic ones, the following 
fundamental equation holds:
\begin{equation}\label{comp_der}
\RR(h_\ast f_\ast)=\RR h_\ast \RR f_\ast.
\end{equation}
\par
Cup products involve tensor products\index{tensor product}; thus, we require
that our abelian categories are equipped with a tensor product with the usual
properties. More precisely, we require that $\scrA$ and $\scrB$ are
\emph{symmetric monoidal categories}\index{monoidal category} with respect to a
right-exact tensor product, e.\,g.\ \cite{MacL}, VII, sections 1 and 7.
\par
Since the main application of cup products is sheaf cohomology, we may not use
projective resolutions. However, it is sufficient to consider \emph{flat
objects}\index{flat object} and flat resolutions to construct the derived tensor
product: Let us call an object $P\in\scrA$ flat, if $P$ is acyclic for the
tensor product. In the category of abelian groups or abelian sheaves this is the
usual notion of flatness.
\par
It is clear that the derived category is again a symmetric monoidal category
with respect to the derived tensor product. It makes things easier, if we
require that the category $\scrA$ has finite $\tor$-dimension (or finite weak
dimension, as in \cite{ce})\index{Tor-dimension}\index{weak dimension}, i.\,e.\
we require that there exists $n\in\bbN$ such that every object has a flat
resolution of length $\leq n$. This ensures that the derived tensor product does
not leave the derived category of complexes that are bounded from below,
$\DD^+\!\scrA$. For details, see proposition \ref{finite_res}.
\par
Crucial to our construction is the notion of \emph{adjoint
functors}\index{adjoint functor}. Let us recall that a pair of functors
$(f^\ast,f_\ast)$ is called \emph{adjoint}, if there exists a functorial
isomorphism
\[
\hom(f^\ast A,B)\isoto\hom(A,f_\ast B)
\]
for every pair of objects, $A$, $B$. The property of adjointness has many useful
consequences: The right adjoint functor $f_\ast$ preserves all limits and the
left adjoint functor $f^\ast$ preserves all colimits (\cite{MacL}, V, 5, theorem
1); in particular, $f_\ast$ is left exact and $f^\ast$ is right exact. If
$f^\ast$ is even exact, then $f_\ast$ preserves injective objects. Moreover,
every pair of adjoint functors induces functorial adjunction
morphisms\index{adjunction morphism} $f^\ast f_\ast\to \id$ and $\id\to f_\ast
f^\ast$.
\par
When we speak of a functor $f\colon\scrA\to\scrB$ of monoidal categories, we do
not require that $f$ is compatible with the tensor product.

\begin{defn}[Cup-Pairing of Functors]\index{cup-pairing}\index{cap-pairing}\index{pairing of functors}
Let $f$, $g$ and $h$ be three functors of monoidal categories $\scrA\to\scrB$.
\begin{enumprop}
  \item A \emph{cup-pairing of functors} $f\cup g\to h$ is a functorial morphism
   \[
   f A\tensor g B\longto h(A\tensor B)
   \]
   for every pair of objects $A,B\in\scrA$.
  \item A \emph{cap-pairing of functors}\footnote{Although the cup-pairing or the cup-product is
     related to the cup-product of algebraic topology, we have chosen the notion of a
     ``cap-pairing'' for duality reasons. It has nothing to do with the cap-product of algebraic
     topology.} $h\to f\cap g$ is a functorial morphism 
    \[
   h(A\tensor B) \longto  fA\tensor gB
   \]
   for every pair of objects $A,B\in\scrA$.
\end{enumprop}
\end{defn}

If $f\cup g\to h$ is a cup-pairing, we call the corresponding functorial
morphism $f A\tensor gB\to h(A\tensor B)$ the \emph{cup-pairing morphism}.
Accordingly, we call the corresponding morphism of a cap-pairing the
\emph{cap-pairing morphism}. The central lemma in our further argumentation is
the following:

\begin{lemma}\label{pair_copair}
Let $f_\ast$, $g_\ast$ and $h_\ast$ be right adjoint to $f^\ast$, $g^\ast$ and $h^\ast$. Then
every cup-pairing of functors $f_\ast\cup g_\ast\to h_\ast$ induces a cap-pairing of functors
$h^\ast\to f^\ast\cap g^\ast$ and vice versa. These constructions are inverse to each other.
\end{lemma}

\begin{proof}
The adjunctions $\id\to f_\ast f^\ast$ and $\id\to g_\ast g^\ast$ induce
together with the pairing morphism the morphism $A\tensor B\to f_\ast f^\ast
A\tensor g_\ast g^\ast B\to h_\ast (f^\ast A\tensor g^\ast B)$. Since $h^\ast$
and $h_\ast$ are adjoint, this composition induces the cap-pairing morphism
$h^\ast(A\tensor B)\to f^\ast A\tensor g^\ast B$. By ``adjoint'' arguments,
every cap-pairing morphism induces a cup-pairing morphism of functors. A lengthy
calculation shows that these constructions are inverse to each other; it uses
arguments as in \cite{MacL}, IV, 1, proof of theorem 2, part (v).
\end{proof}

With these preparations, we can isolate a class of functors that admit cup-products:

\begin{defn}[$\cup$-functor]\index{$\cup$-functor}\index{pre-$\cup$-functor}\label{cup_fn}
Let $f_\ast\colon\scrA\to\scrB$ be a functor of monoidal abelian categories. It is called a 
\emph{pre-$\cup$-functor}, if
\begin{items}
  \item[$(\cup_1)$] $f_\ast$ has a left adjoint $f^\ast$ with the following properties:
  \item[$(\cup_2)$] $f^\ast$ is exact and preserves flat objects.
\end{items}
A pre-$\cup$-functor $f_\ast$ together with a cup-pairing $f_\ast\cup f_\ast\to f_\ast$ is called a
\emph{$\cup$-functor}. 
\end{defn}

In the case of a $\cup$-functor, we will omit the pairing in the notation. In
our applications, it will always be clear, which pairing the functor is equipped
with.

\begin{notation}
Whenever $f_\ast$ is a pre-$\cup$-functor, let us denote by $f^\ast$ its left adjoint.
\end{notation}

The $\cup$-functors we have defined enjoy the following properties: Due to
$(\cup_1)$, the functor $f_\ast$ is left exact; due to $(\cup_2)$, it preserves
injective objects. Hence, compositions of $\cup$-functors can be described by
the Grothendieck spectral sequence. Moreover, since $f^\ast$ is exact, it
coincides with its derived functor $\LL f^\ast$, i.\,e., by $X\mapsto f^\ast X$,
the functor $f^\ast$ defines a well-defined functor of derived categories.
\par
Let us continue with some general adjointness properties in the derived category:

\begin{prop}\label{der_adj0}\label{der_adj}\index{adjoint functor}
Let $f_\ast$ be a pre-$\cup$-functor with left adjoint $f^\ast$. Then, there are
the following canonical isomorphisms:
\begin{enumprop}
   \item $\rhom(f^\ast X,Y)\isoto\rhom(X,\RR f_\ast Y)$.
   \item $\hom(f^\ast X,Y)\isoto\hom(X,\RR f_\ast Y).$
\end{enumprop}
\end{prop}

\begin{proof}
Since $f^\ast$ and $f_\ast$ are adjoint and $f_\ast$ preserves injective objects, the first
assertion is \eqref{comp_der}. Moreover, we have $\hom(X,Y)=\HH^0(\rhom(X,Y))$, hence \enum 2.
\end{proof}


We need the following simple lemma to prove our main theorem:

\begin{lemma}\label{cup_copair}
Let $f^\ast$, $g^\ast$ and $h^\ast$ be exact functors $\scrA\to\scrB$ with the
property that $g^\ast$ respects flat objects. Then every cap-pairing $h^\ast\to
f^\ast\cap g^\ast$ extends to a cap-pairing in the derived category, that is,
the cap-pairing induces a functorial morphism
\[
h^\ast(X\tensorl Y) \longto f^\ast X\tensorl g^\ast Y
\]
for complexes $X,Y\in\DD^+\!\scrA$.
\end{lemma}

\begin{proof}
Let $P\to Y$ be a quasi-isomorphism with a complex $P$ consisting of flat
objects. Since $g^\ast$ respects flat objects, we have morphisms
$h^\ast(X\tensorl Y)= h^\ast(X\tensor P)\to f^\ast X\tensor g^\ast P = f^\ast
X\tensorl g^\ast Y$.
\end{proof}

\begin{thm}[Cup product]\label{der_cup_prod}\index{cup product}
\varhspace Let $f_\ast\cup g_\ast\to h_\ast$ be a cup-pairing of pre-$\cup$-functors
$\scrA\to\scrB$.
\begin{enumprop}
  \item There exists a cup-pairing $\RR f_\ast \cup\RR g_\ast \to\RR h_\ast$, i.\,e.\ there exists
  a functorial morphism, the \emph{derived cup product},
  \[
  \RR f_\ast X\tensorl \RR g_\ast Y\longto\RR h_\ast(X\tensorl Y),
  \]
  for complexes $X$ and $Y$ in the derived category. 
  \item This morphism induces the \emph{cup product}
  \[
   \cup\colon\RR^r f_\ast A\tensor\RR^sg_\ast B\longto\RR^{r+s}h_\ast (A\tensor B)
   \]
   for objects $A,B\in\scrA$ and every pair of integers $r,s\in\bbN$, with the following
   properties: 
   \item for $r=s=0$, this is the pairing of functors we started
   with and it is compatible with exact sequences in the following sense:
   \item Let $0\to A'\to A\to A''\to 0$ be an exact sequence in $\scrA$ let $B\in\scrA$ be
    another object. There is a commutative diagram
    \begin{diagram}[width=0.7in,tight,LaTeXeqno]\label{cup_exact_seq}
    \RR^rf_\ast A'' \tensor \RR^sg_\ast B & \rTo & \RR^{r+s} h_\ast(A''\tensor B) \\
    \dTo                                 &      & \dTo \\
    \RR^{r+1}f_\ast A' \tensor \RR^sg_\ast B & \rTo & \RR^{r+s+1} h_\ast(A'\tensor B)
    \end{diagram}
   and a diagram, which commutes up to sign, with interchanged roles of the exact sequence and $B$. 
\end{enumprop}
By the properties \enum 3 and \enum 4, the cup product is uniquely determined.
\end{thm}

\begin{proof}
The cup-pairing of $f_\ast\cup g_\ast\to h_\ast$ induces a cap-pairing
$h^\ast\to f^\ast\cap g^\ast$ by lemma \ref{pair_copair}. Using lemma
\ref{cup_copair}, this cap-pairing extends to the derived category. By lemma
\ref{pair_copair}, the derived cap-pairing induces the cup-pairing $\RR
f_\ast\cup \RR g_\ast\to \RR h_\ast$. This is the derived cup product. Let us
now take the $(r+s)$-th cohomology of the derived cup product. Together with the
canonical morphism $A\tensorl B\to A\tensor B$ this induces
\[
\HH^{r+s}(\RR f_\ast A\tensorl \RR g_\ast B)\longto \RR^{r+s} h_\ast(A\tensorl
B)\longto\RR^{r+s}h_\ast(A\tensor B).
\]
There exists a canonical morphism $\alpha$, \cite{ce}, IV, \S6, proposition 6.1:
\begin{equation}\label{morph_alpha}
\alpha\colon\RR^rf_\ast A\tensor\RR^sg_\ast B\longto\HH^{r+s}(\RR f_\ast A\tensorl \RR g_\ast B).
\end{equation}
The composition of those morphisms is the cup product. This is the construction
of the cup product. Now, let us prove its properties: For $r=s=0$, we show that
this morphism coincides with the cup-pairing we started with: Let $A,B\in\scrA$
and consider the following commutative diagram in the derived category
\begin{diagram}[width=0.7in,tight]
h^\ast(f_\ast A\tensor g_\ast B)  & \rTo & f^\ast f_\ast A\tensor g^\ast g_\ast B  & \rTo & A\tensor B \\
\uTo                             &      & \uTo                                    &      & \uTo \\
h^\ast(f_\ast A\tensorl g_\ast B) & \rTo & f^\ast f_\ast A\tensorl g^\ast g_\ast B & \rTo &
A\tensorl B\phantom. \\
\dTo                             &      & \dTo                                    &      & \dEqual \\
h^\ast(\RR f_\ast A\tensorl \RR g_\ast B) & \rTo & f^\ast \RR f_\ast A\tensorl g^\ast \RR g_\ast B & \rTo & A\tensorl B. %
\end{diagram}
It is induced by the canonical morphisms ``$\tensorl \to \tensor$'' and $f_\ast
\to \RR f_\ast$ (and for $g_\ast$ accordingly). Again, we use the fact that the
exact functors $f^\ast$, $g^\ast$ and $h^\ast$ define functors on the level of
derived categories. Since $\RR h_\ast$ is right adjoint to $h^\ast$, this
diagram induces
\begin{diagram}[width=0.7in,tight]
f_\ast A \tensor g_\ast B  & \rTo & \RR h_\ast (A\tensor B) \\
\uTo>\phi                  &      & \uTo \\
f_\ast A \tensorl g_\ast B & \rTo & \RR h_\ast (A\tensorl B)\phantom. \\
\dTo                       &      & \dEqual \\
\RR f_\ast A \tensorl \RR g_\ast B & \rTo & \RR h_\ast (A\tensorl B). %
\end{diagram}
If we take the $0$-th cohomology, the morphism $\phi$ induces an isomorphism,
and we see that for $r=s=0$ the cup product morphism of \enum 2 coincides with
the cup-pairing we started with. Finally, we show that the cup product is
compatible with exact sequences. Every exact sequence induces a triangle $A'\to
A\to A''\to A'[1]$. This triangle induces the following morphism of triangles:
\begin{diagram}[width=0.7in,tight]
\RR f_\ast A'\tensorl\RR g_\ast B & \rTo & \RR h_\ast(A'\tensorl B) \\
\dTo                              &      & \dTo \\
\RR f_\ast A\tensorl\RR g_\ast B & \rTo & \RR h_\ast(A\tensorl B) \\
\dTo                              &      & \dTo \\
\RR f_\ast A''\tensorl\RR g_\ast B & \rTo & \RR h_\ast(A''\tensorl B) \\
\dTo                              &      & \dTo \\
\RR f_\ast A'[1]\tensorl\RR g_\ast B & \rTo & \RR h_\ast(A'[1]\tensorl B).
\end{diagram}
By means of the morphism $\alpha$ and the canonical morphism $``\tensorl\to\tensor$'', the last
square induces a diagram 
\begin{diagram}[width=0.7in,tight]
\RR^rf_\ast A'' \tensor \RR^sg_\ast B & \rTo & \RR^{r+s} h_\ast(A''\tensor B) \\
\dTo                                 &      & \dTo \\
\RR^{r+1}f_\ast A' \tensor \RR^sg_\ast B & \rTo & \RR^{r+s+1} h_\ast(A'\tensor B).
\end{diagram}
To prove that the cup product is uniquely determined by these properties,
consider a sequence $0\to A\to I\to Z\to 0$ with an $f_\ast$-acyclic object $I$.
Since $\RR^rf_\ast I=0$ for $r\geq 1$, the left vertical morphism is bijective
for $r\geq 1$ and surjective for $r=0$. Inductively, we can infer that the cup
product morphisms are uniquely determined by property \enum 3 and \enum 4. 
\end{proof}

\begin{remark}
In the category of sheaves, we will study another possibility to construct the
morphism $\alpha$. That is why we do not study the morphism $\alpha$ in detail,
here. In fact, the existence proof for $\alpha$, which is given in \cite{ce}, is
categorial.
\end{remark}

The proof of the theorem shows that the cup product is a natural morphism. This
fact is not visible if one constructs cup products as in \cite{ec}, V, \S1,
proposition 1.16. Moreover, the fundamental properties of cup products as in
\enum 3 and \enum 4 turn out to be formal properties of a morphism in the
derived category. Finally, we show that composition of $\cup$-functors is
compatible with composition of cup products:

\begin{prop}\label{cup_comp0}\index{composition!of functors}\index{composition!of cup products}
Let $h_\ast$, $f_\ast$ and $g_\ast$ be $\cup$-functors with $g_\ast=h_\ast
f_\ast$. If their cup-pairing morphisms are compatible, i.\,e., if the canonical
diagram
\begin{diagram}[midshaft,LaTeXeqno]\label{precup_compose0}
g_\ast A\tensor g_\ast B                   & \rTo &                                    & &
g_\ast(A\tensor B) \\
\dEqual                                    &      &                                    &      & \dEqual \\
(h_\ast f_\ast A)\tensor (h_\ast f_\ast B) & \rTo & h_\ast (f_\ast A \tensor f_\ast B) & \rTo & h_\ast f_\ast  (A\tensor B) %
\end{diagram}
commutes, then the cup-products of $\RR h_\ast$, $\RR f_\ast$ and $\RR g_\ast$
are compatible, that is the following diagram commutes:
\begin{diagram}[midshaft,l>=0.3in]
\RR g_\ast X \tensorl \RR g_\ast Y                               &               &
\rTo^{\cup_g} &
              & \RR g_\ast(X\tensorl Y) & \qquad \\
\dEqual                                                          &               &                 &               & \dEqual \\
(\RR h_\ast\RR f_\ast X) \tensorl (\RR h_\ast\RR f_\ast Y) & \rTo^{\cup_h} & \RR h_\ast(\RR f_\ast
X\tensorl \RR f_\ast Y) & \rTo^{\RR h_\ast\cup_f} & \RR h_\ast\RR f_\ast(X\tensorl Y).
\end{diagram}
\qed
\end{prop}

\begin{cor}
Let $h_\ast$ and $f_\ast$ be $\cup$-functors. If the composition $g_\ast\defeq h_\ast f_\ast$
exists, then this composition is a $\cup$-functor and it is compatible with cup products, i.\,e.,
the conclusion of the above proposition holds.
\end{cor}

\begin{proof}
Since $h_\ast$ and $f_\ast$ have exact left adjoints $f^\ast$ and $h^\ast$, the
composition $g_\ast$ has an exact left adjoint $g^\ast\defeq f^\ast h^\ast$,
which clearly respects flat objects. The cup pairings for $h_\ast$ and $f_\ast$
induce a cup pairing for $g_\ast$ by diagram \eqref{precup_compose0}.
\end{proof}

In the case that the cup-pairing-morphism of the composition $g_\ast=h_\ast
f_\ast$ is compatible with the cup-pairing morphisms of the functors $h_\ast$
and $g_\ast$, we will say that the cup-pairings of $h_\ast$, $f_\ast$ and
$g_\ast$ are compatible. As an important computational tool we can construct the
following pairing of spectral sequences:

\begin{cor}\label{cup_comp}\index{Grothendieck spectral sequence}
Let $h_\ast$, $f_\ast$ and $g_\ast=h_\ast f_\ast$ be three $cup$-functors with compatible
cup-pairings. Then there is the following commutative diagram of pairings of spectral sequences:
\begin{diagram}[l>=0.4in]
\RR^{r_1} h_\ast\RR^{s_1} f_\ast A & \;\times\; & \RR^{r_2} h_\ast\RR^{s_2} f_\ast B & \rTo &
\RR^{r_1+r_2} h_\ast\RR^{s_1+s_2} f_\ast (A\tensor B) \\
\dImplies & & \dImplies && \dImplies \\
\RR^{r_1+s_1} g_\ast A & \times & \RR^{r_2+s_2} g_\ast B & \rTo & \RR^{r_1+s_1+r_2+s_2} g_\ast
(A\tensor B)
\end{diagram}
for objects $A$ and $B$, where the horizontal morphisms are induced by cup products. \qed
\end{cor}


\subsection{Flat Sheaves}\label{flat_sheaves}

In this section, we give a brief survey on flat sheaves and flat resolutions for
the category of abelian sheaves on an arbitrary Grothendieck topology. In
contrast to the exposition of flat sheaves e.\,g.\ in \cite{rd} or \cite{ks} we
will not use stalks. 
\par
 In the following, we will use the tensor product both in the category $\scrS$
of sheaves and in the category $\scrP$ of presheaves. Although we will use the
same notation, it will be clear in which category we are working and, therefore,
which tensor product we refer to. Remember that the sheaf tensor product is
obtained from the presheaf tensor product by sheafification: If
$s\colon\scrP\to\scrS$ denotes the sheafification functor and if
$i\colon\scrS\to\scrP$ denotes the inclusion functor, we have $A\tensor
B=s(iA\tensor iB)$ for sheaves $A$ and $B$.\index{tensor product!of
sheaves}\index{tensor product!of presheaves}
\par
To begin with, recall that an abelian group $P$ is flat if and only it is
torsion free (\cite{ca}, I, \S 2.5, prop. 3). This fact generalises to the
category of abelian sheaves and presheaves.

\begin{defn}[Torsion Free Sheaf]\index{torsion free sheaf}
\varhspace Let $A$ be a sheaf (or presheaf) of abelian groups. It is called
\emph{torsion free} if the multiplication morphism $n\colon A\to A$ is a
monomorphism for every integer $n\geq 1$.
\end{defn}

Obviously, a (pre)sheaf $A$ is torsion free if and only if the groups $A(U)$ are
torsion free for every $U$. This is the key to the following proposition:

\begin{prop}\label{sheaf_flat_tf}\index{flat sheaf}
Let $P$ be an abelian sheaf or a presheaf. Then $P$ is flat if and only if it is torsion free.
\end{prop}

\begin{proof}
Let $P$ be a flat sheaf or a flat presheaf. Tensoring it with the exact sequence
$0\to \bbZ\to\bbZ\to\bbZ\mod n\to 0$ shows that it is torsion free. To prove the
other implication, recall that the group of sections over $U$ of the tensor
product of presheaves $A$ and $B$ is defined to be $A(U)\tensor B(U)$.
Therefore, the characterisation of flat groups as torsion free groups implies
the converse implication in the category of presheaves. It remains to show the
implication ``torsionfree'' $\then$ ``flat'' for sheaves: let $P$ be a torsion free sheaf. The functor
$-\tensor P$ is equal to $s(i-\tensor iP)$. This functor preserves
monomorphisms. Hence $P$ is flat.
\end{proof}

This simple characterisation of flatness is the key to the following assertions
about flat sheaves of abelian groups.

\begin{cor}\label{prop_flat1}
Let $P$ be a flat sheaf (or presheaf) of abelian groups.
\begin{enumprop}
  \item Every subsheaf of $P$ is flat,
  \item the sheafification of $P$ is flat, and
  \item $\HH^0(X,P)$ is a flat group for every object $X$ of the site.
\end{enumprop}
\end{cor}

\begin{proof}
Using the characterisation of flat sheaves, it is clear that a subsheaf of a flat sheaf is flat.
Let $P$ be a flat presheaf and let $T$ be any sheaf of abelian groups. The calculation
\begin{align*}
\hom_\scrS(sP\tensor F,T) &= \hom_\scrS(sP,\shom(F,T)) \\
        &= \hom_\scrP(P,\shom(iF,iT)) \\
        &= \hom_\scrP(P\tensor iF,iT) \\
        &= \hom_\scrS(s(P\tensor iF),T)
\end{align*}
shows that the functor $(sP\tensor -)=s(P\tensor i-)$ respects monomorphisms. Hence, $sP$ is flat.
Finally, \enum 3 is trivial.
\end{proof}

Let $f$ be an arbitrary morphisms of sites. In our applications we are
interested in the case of the morphism ``change of topology'' (e.\ g.\ $R_\fl\to
R_\et$) and in the case of ``restriction of site'' (e.\ g.\ the morphism induced
by $\spec K\into\spec R$) These morphisms induce functors $f^\ast$ and
$f_\ast$. We want to study their flatness preserving properties.

\begin{lemma}\label{indlim_flat}
Let $(P_i)_{i\in I}$ be a directed system of flat groups. The group $\indlim
P_i$ is flat.
\end{lemma}

\begin{proof}
Since the inductive limit is exact and commutes with tensor products, the
functor $-\tensor\indlim P_i$ is isomorphic to the exact functor
$\indlim(-\tensor P_i)$.
\end{proof}

Of course, the same is true for a directed limit of flat sheaves.

\begin{prop}\label{flat_preserve}
Let $P$ be a flat sheaf, then $f^\ast P$ and $f_\ast P$ are flat.
\end{prop}

\begin{proof}
Due to proposition \ref{sheaf_flat_tf}, the assertion for the case of the
functor $f_\ast$ is trivial. Things are more complicated for the case of the
functor $f^\ast$ induced by a morphism $f\colon S\to S'$ of sites. Let $P$ be a
flat sheaf on $S'$. In a first step to obtain $f^\ast P$ from $P$, we have to
extend the sheaf $P$ to a presheaf on $S$. Then we sheafify this presheaf to get
the sheaf $f^\ast P$: The presheaf-extension to $S$ is given by the presheaf
$T\mapsto\indlim P(T)$ for some inductive system, cf.\ \cite{ec}, II, \S2,
prop.\ 2.2 and the following remarks. According to the above lemma, this limit
is a flat group, thus, the presheaf is flat, thus, the associated sheaf is flat.
\end{proof}

The above proof gets easier, if $j\colon Y\to X$ is a restriction map, i.\,e.\
if $j$ belongs to the category of the Grothendieck topology. In this case,
$j^\ast P$ is the restriction of $P$ to $Y$; this sheaf is clearly flat.
\par

In ``classical'' sheaf theory we can argue with stalks\index{stalks} to prove
various properties of flat sheaves, e.\ g.\ \cite{ks}, proposition 2.4.12. In
the general setting, the functor $j_!$ is of great help (cf.\ \cite{ec}, II,
\S3, remark 3.18). Let $j\colon Y\to X$ be a morphism of schemes that belongs to
the category of the Grothendieck topology. It induces a functor $j_!$,
\emph{extension by $0$}, \index{extension by $0$} between presheaf categories
$\scrP_Y\to\scrP_X$ as follows: Let $F$ be a presheaf over $Y$. We set
\[
(j_!F)(T)= \!\!\!\bigoplus_{\phi\in\hom_X(T,Y)}\!\!\! F(T_\phi)
\]
for every $X$-scheme $T$. Here $T_\phi$ is the scheme $T$, regarded as an
$Y$-scheme by means of $\phi$. By sheafification, this functor induces a functor
of the appropriate categories of sheaves. We will denote this functor by $j_!$,
too. It has the following fundamental properties.

\begin{lemma}\label{prop_f!}\
\begin{enumprop}
  \item The functor $j_!$ is left adjoint to $j^\ast$,
  \item $j_!$ is exact (in particular, the functor $j^\ast$ preserves injective objects) and
  \item we have $j_!\bbZ\tensor F=j_!j^\ast F$ for every sheaf $F$. In particular, 
    $j_!\bbZ$ is flat.
\end{enumprop}
\end{lemma}

\begin{proof}
\enum 1 and \enum 2 are \cite{ec}, II, \S3, remark 3.18. A calculation, similar
to the proof of \ref{prop_flat1}, shows that the functor $j_!\bbZ\tensor -$ is
isomorphic to the exact functor $j_!j^\ast$.
\end{proof}


We have now collected all technical facts about flat sheaves that we will need
in the following. Let us continue with flat resolutions and their properties:

\begin{prop}\index{flat resolution}\label{flat_res}
Let $\scrS$ be a category of sheaves of abelian groups and let $F$ be a sheaf.
Then there exists an exact sequence
\[
0\longto P_1\longto P_0\Longto\eps F \longto 0
\]
with flat sheaves $P_0$, $P_1$.
\end{prop}

This lemma generalises the appropriate property for the category of abelian
groups: The $\tor$-dimensions\index{Tor-dimension} \index{weak dimension} (or
weak dimension as in \cite{ce}) of the categories of abelian groups, abelian
presheaves and abelian sheaves are $\leq 1$.

\begin{proof}
For every object $U$ of the Grothendieck topology with structure morphism
$j\colon U\to X$, the sheaf $j_!\bbZ$ is flat. We define an epimorphism of
presheaves and thus of sheaves
\[
\eps\colon P_0\defeq\!\bigoplus_{j\colon U\to R\atop s\in F(U)}\! j_!\bbZ\longmapsto F
\]
by $1\mapsto s\in F(U)$ for the component $(j\colon U\to R,s)$. As the the sheaf
$\bigoplus j_!\bbZ$ is flat, the kernel $P_1\defeq\ker\eps$ it is flat, too.
\end{proof}

\begin{remark}\label{enough_proj}\index{enough projectives}
For later use, let us remark that this proposition and lemma \ref{prop_f!} imply
that the category $\scrP$ of abelian presheaves has enough projectives in form
of the presheaves $j_!\bbZ$: Since the presheaf functor $j^\ast$ is exact, the
presheaf functor $j_!$ preserves projectives. In fact, we have
$\hom(j_!\bbZ,F)=F(U)$; this functor is exact on the category of presheaves.
Since flat sheaves need not be projective, we cannot expect that every sheaf has
a projective resolution of length $1$.
\end{remark}

\begin{prop}\label{finite_res}
Let $\ast\in\set{b,+,-,\emptyset}$ and let $X\in\DD^\ast\!\scrS$ be a complex.
Then, there exists a quasi-isomorphism $P\to X$ with a complex
$P\in\DD^\ast\!\scrS$ consisting of flat sheaves. In particular, for any
complex, bounded from below, there exists a quasi-isomorphism $P\to X$ with a
complex $P$ of flat sheaves that is bounded from below.
\end{prop}

\begin{proof}
This is a special case of a more general result. Cf.\ \cite{ks}, chapter I, exc.\ I.23.
\end{proof}

Using this result, it is easy to work with the derived tensor
product.\index{derived tensor product} Remember that $A\tensorl B$ is obtained
by choosing a quasi-isomorphism $P\to B$ with a complex $P$ consisting of flat
sheaves and setting $A\tensorl B\defeq A\tensor P$. The above proposition
implies that in this process we do not leave the category of complexes bounded
from below. Thus, we have a convergent spectral sequence to compute the
hypercohomology of double complexes obtained from the derived tensor product.
\par
Finally, we want to construct the K\"unneth spectral sequence. Unfortunately,
proofs for the existence of the K\"unneth spectral sequence can only be found
for categories with enough projectives, e.\.g.\ for module categories,
\cite{ec}, XVII, \S3, \enum 4 or \cite{g}, I, \S5, theorem 5.5.1. These proofs
do not generalise without more ado to the category of sheaves.

\begin{prop}\label{kunneth}\index{K\"unneth spectral sequence}
Let $X$ and $Y$ be complexes of sheaves. There exists the homological \emph{K\"unneth spectral
sequence}
\[
E^2_{r,s}=\bigoplus_{s_1+s_2=s} \tor_r(\HH_{s_1}(X),\HH_{s_2}(Y))\ssto E_{r+s}=\HH_{r+s}(X\tensorl
Y).
\]
\end{prop}

\begin{proof}
In a first step, we prove the existence of the K\"unneth spectral sequence in
the category $\scrP$ of presheaves. The proof given in \cite{g}, I, \S 5,
theorem 5.5.1 (or 5.4.1, respectively) is purely categorial and rests on the
existence of (projective) Cartan-Eilenberg resolutions. Since $\scrP$ has enough
projectives\index{enough projectives} (remark on page \pageref{enough_proj}), we
can construct a projective Cartan-Eilenberg resolution\index{Cartan-Eilenberg
resolution} (in \cite{g} those resolutions are called ``r\'esolution injective''
or ``r\'esolution projective'', respectively) in $\scrP$ and construct the
K\"unneth spectral sequence for the complexes of presheaves $iX$ and $iY$ as it
is shown there.
\par
Since the functor of sheafification is exact, and since tensor products and flat
sheaves are compatible with sheafification, we sheafify the K\"unneth spectral
sequence in the category $\scrP$ to obtain a spectral sequence in the category
$\scrS$ of sheaves. Routine verifycations show that this spectral sequence has
the right $E_2$ and limit terms and thus is the K\"unneth spectral sequence in
the category of abelian sheaves.
\end{proof}

\pagebreak[2]

\begin{remarks}
 \item Since the $\tor$-dimension of the categories of abelian sheaves and pre\-sheaves is 1 by
   proposition \ref{flat_res}, the K\"unneth spectral sequence collapses to a family of exact
   sequences
   \[
\!\!\!\!\!\!\!\!\!\!\!\!\!\!
   0\to\!\!\!\bigoplus_{s_1+s_2=s}\!\!\!\HH_{s_1}(X)\tensor\HH_{s_2}(Y)\to \HH_s(X\tensorl Y)\to\!\!\!
   \bigoplus_{s_1+s_2=s-1}\!\!\!\tor_1(\HH_{s_1}(X),\HH_{s_2}(Y))\to 0.
   \]
 \item The first edge morphism induces the morphism $\alpha$ of \eqref{morph_alpha}. For this
 equality cf.\ \cite{ce}, XVII, \S3, proposition 3.1.
\end{remarks}

\subsection{Cup Products in the Category of Abelian Sheaves}

In this section, we study some fundamental functors of sheaf cohomology. It is
an important observation that they all come along as adjoint functors. Using the
theory of $\cup$-functors, we will see that they all admit cup products. In our
applications, we are mainly interested in sheaves on the \'etale and the flat
site of a discrete valuation ring $R$, so let us adopt our notation to this
situation. If nothing else is specified, we are working with a fixed site over
$R$ with its abelian category of abelian sheaves. Of course, all these results
are true for the category of abelian sheaves on any site or topological space.

The following facts are well-known, e.\,g.\ \cite{ec}.

\begin{prop}\label{functor_sheaf}
The following pairs of functors are adjoint:
\begin{enumprop}
  \item Sheafification $s\colon\scrP\to\scrS$ and inclusion
    $i\colon\scrS\to\scrP$ for the category of abelian sheaves on any fixed site.
  \item $f^\ast$ and $f_\ast$ associated to the canonical morphism $f\colon
    R_\fl\to R_\et$. \item $j^\ast$ and $j_\ast$ for a morphism $j\colon Y\to X$
    of schemes. \item The functors ``constant $R$-sheaf''\index{constant sheaf}
    $G\mapsto\tilde G$ and $\HH^0(R,-)$.
  \item The functor $G\mapsto i_\ast\tilde G$ and $\HH^0_k(R,-)\defeq\ker(\HH^0(R,-)\to\HH^0(K,-))$,
    the functor ``sections with support over $k$'' of local cohomology,\footnote{For
    a brief introduction to local cohomology, see e.\,g.\
    \cite{lc}.}\index{local cohomology} with respect to the closed immersion 
    $i\colon \spec k\into \spec R$.
\qed
\end{enumprop}
\end{prop}

It is common to denote the functor $\HH^0(R,-)$ by $\Gamma$ and to denote its
derived functor in the sense of derived categories by $\RR\Gamma$. Moreover, let
us write $\Gamma\!_k$ for the functor $\HH^0_k(R,-)$ and $\RR\Gamma\!_k$ for its
derived functor.
\par
In fact, even the functors of \enum 1 and \enum 4 can be seen as an incarnation
of a functor associated to a morphism of sites: Let $\calC$ be the category of
the Grothendieck topology $\mathfrak T$. Then the category $\calC$ with trivial
coverings defines a Grothendieck topology. Let us denote this topology by
$\calC$ again. There is a canonical morphism $\calC\to\mathfrak T$. This
morphism induces the functors $s$ and $i$. To study the functor of \enum 4,
consider the morphism $h\colon\set\ast\to (\text{Schemes}/R)$, $\ast\mapsto R$.
Then $h^\ast$ is the functor ``constant sheaf'' and $h_\ast$ is the functor
$\HH^0(R,-)$.
\par
There are the following extended adjointness properties for the functors \enum 1
-- \enum 4 and for the tensor product with its adjoint $\shom$.

\begin{prop}\label{rhom_adjoint}\index{adjoint functor}\index{tensor product}\index{tensor
product!derived tensor product} Let $\phi_\ast$ be one of the functors of
proposition \ref{functor_sheaf}, \enum 1 -- \enum 4. There are canonical
isomorphisms
\begin{enumprop}
  \item $\phi_\ast\shom(\phi^\ast X,Y)\isoto\shom(X,\phi_\ast Y)$.
  \item $\hom(X\tensor Y, Z)\isoto\hom(X,\shom(Y,Z))$.
  \item $\shom(X\tensor Y, Z)\isoto\shom(X,\shom(Y,Z))$,
\end{enumprop}
for suitable sheaves or complexes of sheaves. These isomorphisms extend to the
derived category:
\begin{enumprop}
\setcounter{enumiv}{3}
  \item $\RR \phi_\ast\rshom(\phi^\ast X,Y)\isoto\rshom(X,\RR \phi_\ast Y)$.
  \item $\rhom(X\tensorl Y, Z)\isoto\rhom(X,\rshom(Y,Z))$.
  \item $\rshom(X\tensorl Y, Z)\isoto\rshom(X,\rshom(Y,Z))$,
\end{enumprop}
for suitable complexes $X,Y,Z$ in the derived category. Moreover, the functors
$\tensorl$ and $\rshom$ are adjoint:
\begin{enumprop}
\setcounter{enumiv}{6}
  \item $\hom(X\tensorl Y, Z)\isoto\hom(X,\rshom(Y,Z))$.
\end{enumprop}
\end{prop}

To prove this proposition, we need the following well-known lemma

\begin{lemma}\label{shomPI}
Let $P$ be a flat sheaf and let $I$ be an injective sheaf. Then, the sheaf
$\shom(P,I)$ is injective.
\end{lemma}

\begin{proof}
Indeed, by the adjunction formula, the functor $\hom(-,\shom(P,I))$ is
isomorphic to the exact functor $\hom(-\tensor P,I)$; hence, $\shom(P,I)$ is
injective.
\end{proof}

\begin{proofof}{proposition \ref{rhom_adjoint}}
The isomorphisms of \enum 1 -- \enum 3 are well-known: \cite{ec}, II, 3.22 for
\enum 1, ibid.\ 3.19 for \enum 3. Let us continue with the ``derived''
isomorphisms: To show \enum 4, we can assume that $X$ consists of flat, and $Y$
of injective objects. By lemma \ref{shomPI}, the sheaf $\shom(f^\ast X,Y)$ is
injective and the formula is \enum 1. In the same way, we can show \enum 5 and
\enum 6. Let $Z$ be a complex of injective objects and let $Y$ be a complex of
flat objects. Then, the formulas of \enum 5 and \enum 6 are the formulas of
\enum 2 and \enum 3. Finally, by taking the $0$-th cohomology of \enum 6 we get
\enum 7.
\end{proofof}

\begin{prop}\label{cup_functor_sheaf}\index{$\cup$-functor}
The functors of proposition \ref{functor_sheaf} are $\cup$-functors.
\end{prop}

\begin{proof}
It is known that the corresponding left adjoint functors are exact, \cite{ec},
II, introduction to \S3 with II, \S2, 2.6. Using the characterisation of flat
groups and flat sheaves being torsion free, it is clear that the functor
$G\mapsto \tilde G$ preserves flats objects. Using proposition
\ref{flat_preserve}, it follows that the functors $G\mapsto i_\ast \tilde G$,
$f^\ast$ and $j^\ast$ preserve flat objects, too. Finally, we have to construct
canonical cup-pairings. In the case of the functors of \enum 1 -- \enum 4 of
proposition \ref{functor_sheaf}, the functor $\phi^\ast$ even commutes with
tensor products, i.\,e., there exists a canonical cap-pairing morphism
$\phi^\ast (A\tensor B)\isoto\phi^\ast A\tensor \phi^\ast B$. By lemma
\ref{pair_copair}, this isomorphism gives rise to the desired cup-pairing.
\par
Finally, let us show that the functor $\HH^0_k(R,-)$ is equipped with a canonical cup-pairing.
Consider the commutative diagram with exact second line
\begin{diagram}[l>=0.2in,height=0.22in]
  &      & \HH^0_k(R,A)\tensor\HH^0_k(R,B) & \rTo & \HH^0(R,A)\tensor\HH^0(R,B) & \rTo & \HH^0(K,A)\tensor\HH^0(K,B)  \\
  &      & \dDashto  &                     & \dTo &                             & \dTo \\  
0 & \rTo & \HH^0_k(R,A\tensor B)           & \rTo & \HH^0(R,A\tensor B)         & \rTo & \HH^0(K,A\tensor B)
\end{diagram}   
Keep in mind that the first line needs not to be exact. However, it is a
complex. The second and third vertical morphisms are the cup-pairings of
$\HH^0(R,-)$ and $\HH^0(K,-)$.
\end{proof}

Since the functors of proposition \ref{functor_sheaf} are $\cup$-functors, they
admit cup products by theorem \ref{der_cup_prod}:

\begin{thm}\label{cup_prod}\index{cup product}\
Let $\phi_\ast$ be any of the left adjoint functors of proposition
\ref{functor_sheaf}. There is a canonical morphism $\RR \phi_\ast X\tensorl \RR
\phi_\ast Y \to \RR \phi_\ast(X\tensorl Y)$ for complexes $X$, $Y$ in the
derived category $\DD^+\!\scrS_\et$. It induces morphisms
\[
\RR^r\phi_\ast A\tensor \RR^s \phi_\ast B  \longto \RR^{r+s} \phi_\ast(A\tensor B)
\]
for sheaves $A$, $B$ with the properties of theorem \ref{der_cup_prod}. \qed
\end{thm}

\begin{remark}
    By the adjunction of $\RR\!\shom$ and $\tensorl$, the cup product morphism
    is the same as a morphism
   \[
   \RR \phi_\ast A\longto\RR\!\shom(\RR \phi_\ast B,\RR \phi_\ast(A\tensorl B)),
   \]
   which, again, we will refer to as the cup product.
\end{remark}

It is an easy calculation to show that this cup product coincides with the
product that is constructed in \cite{ec}, V, \S1, proposition 1.16 using
Godement resolutions. In fact, we have already proven this fact implicitly in
theorem \ref{der_cup_prod}: Both products share the same properties that
characterise a cup product.
\par

%
%

\subsection{The Cup Product of \Cech\ Cohomology}

In some cases, \Cech\ cohomology can be used to compute the ``true'' derived
functor cohomology groups $\HH^\ast(X,-)$. In this section, we show that in this
cases the cup product, which is defined on \Cech\ cocycles, can be used to
compute the cup product of derived functor cohomology. Remember that the \Cech\
functors $\cHH^\ast(\frakU,-)$ for a covering $\frakU$ of $U$ and
$\cHH^\ast(U,-)$ are universal, cohomological $\delta$-functors on the category
$\scrP$ of presheaves, \cite{a62}, theorem 3.1. For a sheaf $F$, both groups
$\cHH^0(\frakU,iF)$ and $\cHH^0(U,iF)$ coincide with the group $\HH^0(U,F)$.
Thus, the relation between \Cech\ cohomology and derived functor cohomology is
described by the \emph{\Cech\ spectral sequences}\index{Cech spectral sequence}
\begin{equation}\label{cech_ss}
\cHH^r(\frakU,\RR^si F)\ssto\HH^{r+s}(U,F)\qquad\text{and}\qquad\cHH^r(U,\RR^si F)\ssto\HH^{r+s}(U,F)
\end{equation}
for the inclusion functor $i\colon\scrP\to\scrS$. Before we can state the first
lemma, we need some notations: Let $\frakU=(U_i\Longto{j_i} U)_{i\in I}$ be a
covering. By $j_{i_0,\ldots, i_n}$, we denote the morphism
$U_{i_0}\times_U\ldots\times_U U_{i_n}\to U$. For an abelian group $G$, we
denote by $\tilde G$ the constant presheaf\index{constant presheaf} associated
to $G$.

\begin{lemma}\index{projective resolution}
Let $\frakU$ be a covering of $U$. The canonical sequence of presheaves
\[
C_\frakU(\bbZ)\colon \ldots\;\pile{\longto\\\longto\\\longto\\\longto}\;\bigoplus_{i_0,i_1,i_2\in I}
(j_{i_0,i_1,i_2})_!\tilde \bbZ \;\pile{\longto\\\longto\\\longto}\; \bigoplus_{i_0,i_1\in I}
(j_{i_0,i_1})_!\tilde \bbZ \;\pile{\longto\\\longto}\; \bigoplus_{i\in I}(j_{i})_!\tilde \bbZ
\]
is a projective resolution of the cokernel of the latter pair of morphisms in
the category $\scrP$ of presheaves. We will denote this cokernel by
$c_\frakU\bbZ$. This presheaf is flat.
\end{lemma}

\begin{proof}
This is \cite{a62}, proof of theorem 3.1 for the proof that the sequence is
acyclic and the remark on page \pageref{enough_proj} that it consists of
projectives. In \cite{a62}, ibid, it is shown that for $V\to U$ this sequence
induces a canonical, exact sequence
\[
\ldots\longto\bigoplus_{S\times S\times S}\longto \bigoplus_{S\times S}\bbZ\longto\bigoplus_S \bbZ
\]
for a set $S$, depending on $V$. The cokernel of the last map,
$c_\frakU\bbZ(V)$, is trivial if $S$ is empty and it is isomorphic to $\bbZ$
otherwise. Hence, this presheaf is flat.
\end{proof}

In \cite{a62}, proof of theorem 3.1 it is shown that this complex induces the
\Cech\ complex via $\check C(\frakU,F)=\hom(C_\frakU(\bbZ),F)$. This resolution
will play an important role in the following. Moreover, since the functor
$\hom(-,F)$ is left exact, we have 
\begin{equation}\label{cech_represent}
 \cHH^0(\frakU,F)=\HH^0(\check C(\frakU,F))=\hom(c_\frakU\bbZ,F).
\end{equation}
This means that the \Cech\ functor is represented by the presheaf $c_\frakU\bbZ$.

\begin{prop}\index{Cech cohomology}\
Let $\frakU=(U_i\to U)_{i\in I}$ be a covering.
\begin{enumprop}
   \item The \Cech\ functor $\cHH^0(\frakU,-)$ is a $\cup$-functor.
   \item In particular, there exists a cup product for $\cGamma_\frakU\defeq\cHH^0(\frakU,-)$.
   \item This cup product induces a cup product for the \Cech\ functor $\cGamma\defeq\cHH^0(U,-)$.
\end{enumprop}
\end{prop}

\begin{proof}
Let us prove that $\cHH^0(\frakU,-)$ has a left adjoint. We define the functor
``constant presheaf with respect to the covering $\frakU$''\index{constant
presheaf!with respect to a covering}, $c_\frakU G$, for an abelian group $G$ by 
\[
c_\frakU G\defeq \tilde G\tensor c_\frakU\bbZ,
\]
where $\tilde G$ is the constant presheaf associated to $G$. An easy calculation
using equation \eqref{cech_represent} shows that the functor $c_\frakU$ is left
adjoint to the \Cech\ functor. 
\par
Since $c_\frakU\bbZ$ is flat, the functor $c_\frakU$ is exact. Since the
presheaf $c_\frakU\bbZ$ consists of the group $\bbZ$ or $0$ at every level, the
tensor product $c_\frakU\bbZ\tensor c_\frakU\bbZ$ is isomorphic to
$c_\frakU\bbZ$. In particular, there is a canonical isomorphism
\begin{equation}\label{cech_cappairing}
c_\frakU\tensor(\tilde F\tensor \tilde G)\isoto(c_\frakU\bbZ\tensor\tilde F)\tensor
(c_\frakU\bbZ\tensor\tilde G),
\end{equation}
This can be regarded as a cap-pairing $c_\frakU\to c_\frakU\cap c_\frakU$ and,
thus, there is a cup-pairing
$\cGamma_\frakU\cup\cGamma_\frakU\to\cGamma_\frakU$.
\par
These facts imply that $\cGamma_\frakU$ admits a cup product. Now, $\cGamma$ is
defined to be $\indlim_\frakU\cGamma_\frakU$, where the limit varies over all
coverings $\frakU$ of $U$. By this limit we define the cup product for the
\Cech-functor $\cGamma$. Remember that $\indlim$ turns flat sheaves into flat
sheaves, remark after proposition \ref{indlim_flat}. Consequently, taking the
direct limit is compatible with the derived tensor product. This way, we obtain
the cup product for $\cGamma$. 
\end{proof}

We want to describe the cup product of the \Cech\ functor: The cap-pairing
morphism of \eqref{cech_cappairing} for $G=F=\bbZ$ can be extended to a
canonical pairing of the resolutions:
\[
\begin{array}{rcl}
C^{r+s}_\frakU(\bbZ) & \longto & C^r_\frakU(\bbZ)\tensor C^s_\frakU(\bbZ) \\
e_{i_0,\ldots,i_{r+s}} & \longmapsto &  e_{i_0,\ldots,i_r}\tensor e_{i_r,\ldots,i_{r+s}}.
\end{array}
\]
In particular, for $r=s=0$ this pairing induces the isomorphism
$c_\frakU\bbZ\isoto c_\frakU\bbZ\tensor c_\frakU\bbZ$ of the above proof.
Moreover, this pairing induces a pairing of \Cech\ complexes as follows: Let
$a=(a_0,\ldots,a_r)\in C^r(\frakU,A)$ and let $b=(b_0,\ldots,b_s)\in
C^s(\frakU,B)$. Their image in $C^{r+s}(\frakU,F\tensor G)$ is given by the
\emph{Alexander-Whitney formula}: \index{Alexander-Whitney formula}
\begin{equation}\label{alex_whit}
(a\cup b)_{i_0,\ldots i_{r+s}}=a_{i_0,\ldots i_r}\tensor b_{i_r,\ldots i_{r+s}},
\end{equation}
cf.\ \cite{ec}, V, \S1, remark 1.19. This suggests that the cup product of
\Cech\ cohomology is given by the Alexander-Whitney formula and, in the case
that sheaf cohomology is given by \Cech\ cohomology, even the cup product is
given by \eqref{alex_whit}. We are going to prove these facts in the following.

\begin{prop}\label{cech_alexwhit}
The cup product defined on \Cech\ cocycles by the Alexander-Whit\-ney-formula
agrees with the cup product of the $\cup$-functor $\cHH^0(\frakU,-)$.
\end{prop}

\begin{proof}
The complex $C_\frakU(\bbZ)$ is a projective resolution of $c_\frakU(\bbZ)$. Since
$\cHH^0(\frakU,F)=\hom(c_\frakU(\bbZ),F)$, there are canonical isomorphisms
\begin{equation}\label{Cech_RGamma}
\RR\cGamma_\frakU F\isoto\rhom(c_\frakU(\bbZ),F)\isoto  \hom(C_\frakU(\bbZ),F)\isoto
C(\frakU,F).
\end{equation}
The derived cup product of $\cGamma_\frakU$ can therefore be written as the
following morphism in the derived category:
\begin{equation}\label{cech_hom}
\hom(C_\frakU(\bbZ),F) \tensorl \hom(C_\frakU(\bbZ),G)\longto\hom(C_\frakU(\bbZ),F\tensorl G).
\end{equation}
Let us replace $F$ and $G$ by flat, i.\,e.\ torsion free, resolutions. Then, the
complex $\hom(C_\frakU(\bbZ),F)$ consists of direct sums of flat summands
$F(U_i)$ for morphisms $U_i\to U$. Hence, it consists of flat groups.
Consequently, we can replace the derived tensor product $\tensorl$ by the
ordinary tensor product $\tensor$. Consequently, the morphism is given by the
Alexander-Whitney formula.
\end{proof}

Let us emphasise that equation \eqref{Cech_RGamma} states that the derived
functor $\RR\Gamma_\frakU$ of the \Cech\ functor actually \emph{is} the \Cech\
complex relative to the covering $\frakU$.

\begin{remark}\index{group cohomology}\index{cup product!of group cohomology}
Using the same arguments, we can show that cup products of group cohomology can
equivalently be defined by $\cup$-functors and by the Alexander-Whitney formula
on the Bar resolution, cf.\ \cite{Br}, V.3.
\end{remark}

By \Cech\ cohomology, we refer to the cohomology with respect to the functors
$\cHH^\ast(\frakU,-)$ and $\cHH^\ast(U,-)$. We say that \Cech\ cohomology and
derived functor cohomology agree in dimension $r$ if the edge morphism
\[
\cHH^r(\frakU,F)\longto\HH^r(U,F)\qquad\text{or}\qquad\cHH^r(U,F)\longto\HH^r(U,F)
\]
of the \Cech\ spectral sequence \eqref{cech_ss} is an isomorphism. For example,
this is the case for the latter edge morphism for $r=0,1$ .

\begin{thm}\label{cech_true_cup}
If \Cech\ cohomology agrees with derived functor cohomology, the cup product of
\Cech\ cohomology agrees with the cup product of derived functor cohomology. In
particular, it can be calculated by the Alexander-Whitney-formula.
\end{thm}

\begin{proof}
For the case of $\cHH^\ast(\frakU,-)$, this is corollary \ref{cup_comp} and
proposition \ref{cech_alexwhit}. In the case of $\cHH^\ast(U,-)$, we take the
limit over all coverings $\frakU$ over the diagram of proposition
\ref{cup_comp0} and continue as in corollary \ref{cup_comp}.
\end{proof}

For use in the next section, let us note

\begin{cor}\label{cech1_cup}
The cup product $\HH^1(U,A)\times\HH^0(U,B)\to\HH^1(U,A\tensor B)$ can be
computed by \Cech\ cohomology. \qed
\end{cor}

We will see an explicit example for such a computation in the proof of proposition
\ref{cup_description}.

%% file: bester.tex
\section{The $p$-Part of Grothendieck's Pairing in the Case of Semistable Reduction}

\subsection{Rigid Uniformisation}

In the following, let $R$ be a complete, discrete valuation ring of equal
characteristic $p\neq 0$ with algebraically closed residue class field $k$. Let
$K\defeq\quot R$ be its field of fractions and let $A_K$ be an abelian variety
over $K$ with dual abelian variety $A_K'$.
\par
In this section, we summarise some results on rigid uniformisation of abelian
varieties as we need them in the following. These results are due to \cite{R},
\cite{bl} for rigid uniformisation and \cite{BX96} for formal N\'eron models.
For the following proposition, it is crucial that $K$ is complete:

\begin{prop}[Rigid Uniformisation]\label{rigid_uniform}\index{rigid uniformisation}
Let $R$ and $K$ be as above and let $A_K$ be an abelian variety over $K$ with
semi\-stable reduction.
\begin {enumprop}\index{good reduction}
  \item There exists a $K$-group scheme $E_K$, which is an extension of an abelian variety $B_K$
    with good reduction by a split torus $T_K$ of dimension $d$, such that in
    the category of rigid $K$-groups the abelian variety $A_K$ is isomorphic to
    a quotient of $E_K$ by a split lattice $M_K$ of rank $d$. This means, $M_K$
    is a closed analytic subgroup of $E_K$, which is isomorphic to the constant
    $K$-group scheme $\bbZ^d$. This data can be summarised by the following
    exact   sequences:
    \begin{align*}    
    0\longto M_K\longto E_K\longto A_K\longto 0 &\quad\text{and} \\
    0\longto T_K\longto E_K\longto B_K\longto 0
   \end{align*}    
   where the morphism $E_K\to A_K$ exists in the category of rigid $K$-groups.
  \item Let $A'_K$ be the dual abelian variety whose uniformisation is denoted by $M'_K$,
    $E'_K$, $T'_K$. Then, $M'_K$ is the character group of $T_K$ (and vice
    versa). The abelian varieties $B_K$ and $B'_K$ are dual to each other.
\qed
\end{enumprop}
\end{prop}

The exact sequence of rigid uniformisation extends to the level of formal
N\'eron models as follows:

\begin{prop}\label{uniform_neron}\index{formal N\'eron model}\
\begin{enumprop}
  \item The rigid uniformisation above gives rise to an exact sequence of formal N\'eron models:
   \[
   0\longto M\longto E\longto A\longto 0.
   \]
   The induced morphism $E^0\to A^0$ is an isomorphism.
  \item This sequence induces the following exact sequence of groups of components
   \[
    0\longto M\longto\oldphi_E\longto\oldphi_A\longto 0.
   \]
\end{enumprop}
\end{prop}

\begin{proof}
See \cite{BX96}, theorem 2.3 and proposition 5.4.
\end{proof}

To simplify notation, let us write $M$ for both the N\'eron model of $M_K$ and
the (abstract) group which $M_K$ is associated to.
\par

We will need the following notions of duality: Let $G$ be any (abstract) abelian
group. We define $G^\vee$ to be the group $\hom(G,\bbZ)$, and $G^\ast$ to be the
Pontrjagin dual $\hom(G,\bbQZ)$. If $G$ is finite, there is a canonical
isomorphism $G^\ast\isoto\ext^1(G,\bbZ)$, induced by the $\ext$-sequence of
$0\to\bbZ\to\bbQ\to\bbQZ\to 0$.

\begin{remark}
Taking the uniformisation of the dual abelian variety $A_K'$ into consideration,
there is the bijective map \emph{evaluation of characters}
\begin{equation}\label{eval_char}
\begin{array}{cl}
  \oldphi_{T'} & \longto M^\vee=\hom(\hom(T'_K,\bbG_m),\bbZ),\\
  \bar t & \longmapsto(\phi\mapsto\nu(\phi(t)))
\end{array}
\end{equation}
for an element $\bar t\in\oldphi_{T'}=T'(R)\mod T'^0(R)=T'_K(K)\mod T'^0(R)$ and
the discrete valuation $\nu$ on $K$. This map will play a central role in the
following; we will denote it by $\nu^\ast$.
\end{remark}

\subsection{The Monodromy Pairing}\label{sect_monodromy}\index{monodromy pairing}

In \cite{sga7}, Grothendieck suggested comparing the perfect monodromy pairing
to Grothendieck's pairing of groups of components in order to show that
Grothendieck's pairing is perfect in the case of semistable
reduction\index{semistable reduction}. This is done in \cite{We}. We will sketch
the main ideas without giving proofs. In the following, let $A_K$ be an abelian
variety with semistable reduction. Let $M$, $E$ etc.\ be the objects of rigid
uniformisation we obtained in proposition \ref{uniform_neron}.
\par
We have two exact sequences of component groups 
\begin{equation}\label{comp_groups}
\begin{array}{cl}
0\longto M'\longto\oldphi_{E'}\longto\oldphi_{A'}\longto 0\phantom. &\quad\text{and} \\
0\longto M\longto\oldphi_{E}\longto\oldphi_{A}\longto 0.
\end{array}
\end{equation}
Again, we write $M$ for the abstract group $M_K$ is associated to. There is a
canonical isomorphism $\oldphi_T\isoto\oldphi_E$, \cite{BX96}, theorem 4.11,
and, thus, there are isomorphisms 
\begin{equation}\label{phiET_M'}
\oldphi_E\isoto\oldphi_T\isoto M'^\vee.
\end{equation}
The last map is the map ``evaluation of characters'', $\nu^\ast$, as in
\eqref{eval_char}.

\begin{prop}\label{mp_eval_cha}
The isomorphisms obtained from \eqref{phiET_M'} induce a commutative diagram
\begin{diagram}[width=0.5in,tight]
M & \rTo & \oldphi_{E} & \rTo & \oldphi_{A} & \rTo & 0\phantom. \\
\dTo>{\!\wr} && \dTo>{\!\wr} && \dDashto>\sigma \\
\oldphi_{E'}^\vee & \rTo & M'^\vee & \rTo & \ext^1(\oldphi_{A'},\bbZ) & \rTo & 0.
\end{diagram}
In particular, the induced map $\sigma$ is bijective. Hence, it induces a
perfect pairing $\oldphi_A\times\oldphi_{A'}\to\bbQZ$.
\end{prop}

\nopagebreak 
We will refer to this pairing of component groups as the \emph{monodromy
pairing}.

\begin{proof}
The exact sequence in the first line is clear. The second one is obtained by
applying $\hom(-,\bbZ)$ to the first sequence of \eqref{comp_groups}; it yields
the exact sequence
\[
\hom(\oldphi_{E'},\bbZ)\longto\hom(M',\bbZ)\longto\ext^1(\oldphi_{A'},\bbZ)\longto\ext^1(\oldphi_{E'},\bbZ).
\]
Since $\oldphi_{E'}$ is free, the latter group is trivial. Hence, we have the
exact sequence of the second line. The commutativity of the first square is
shown in \cite{We}, section 3 and 4. It induces the morphism $\sigma$ that is
clearly bijective. Since the group $\ext^1(\oldphi_{A'},\bbZ)$ is isomorphic to
$\hom(\oldphi_{A'},\bbQZ)$, the isomorphism $\sigma$ induces a perfect pairing
of groups of components.
\end{proof}

\begin{prop}\label{mp_gp}\label{mp_gp2}\index{Grothendieck's pairing}
In the situation of semistable reduction, the monodromy pairing and
Grothen\-dieck's pairing of groups of components coincide up to sign. That is,
the following diagram commutes up to sign
\begin{diagram}
\oldphi_{T} & \rTo & \oldphi_{A}\\
\dTo>{\nu^\ast} && \dTo>\gp \\
M'^\vee & \rTo & \oldphi_{A'}^\ast &.
\end{diagram}
In particular, in the case of an abelian variety with semistable reduction,
Grothen\-dieck's pairing is perfect.
\end{prop}

\begin{proof}
See \cite{We}, proposition 5.1.
\end{proof}

\subsection{Bester's Pairing}

In \cite{bes}, M.\ Bester constructed Bester's pairing using \emph{local
cohomology}. Let us review the theory of local cohomology in the framework of
\'etale cohomology. (cf.\ \cite{ec} pages 73--78)
\par
The topological space $\spec R$ consists of two points and there are two
immersions
\[
\spec k\Longto i \spec R\Longot j \spec K,
\]
the first being a closed immersion, the second being an open immersion. We
consider the functor $\HH^0_k(R,-)$ on the category of abelian sheaves on $\spec
R$ which is defined by the exact sequence
\begin{equation}\label{loc_coh_def}
0\longto \HH^0_k(R,F)\longto\HH^0(R,F)\longto\HH^0(K,F)
\end{equation}
for a sheaf $F$ on the flat or the \'etale site of $\spec R$. Obviously, it is
left exact so we can define its derived functors
$\HH^r_k(R,-)\defeq\RR^r\HH^0_k(R,-)$. Of course, there is a long exact
cohomology sequence for every exact sequence, but there is another important
sequence for every sheaf $F$:

\begin{prop}[Sequence of Local Cohomology]\label{seq_loc_coh}
\varhspace Let $F$ be a sheaf. Then, sequence \eqref{loc_coh_def} is part of the
\emph{long exact sequence of local cohomology}:
\[
\ldots\longto \HH^i_k(R,F)\longto\HH^i(R,F)\longto\HH^i(K,F)\longto\HH^{i+1}_k(R,F)\longto\ldots
\]
for every $i\geq 0$.
\end{prop}

Of course, this works for both the \'etale and the flat topology.

\begin{proof}
Lemma \ref{seq_loc_coh_der} or \cite{bes}, Appendix, p.\ 172.
\end{proof}

As counterpart to local cohomology, Bester constructs a functor $\scrF$. We will
sketch its construction: As always, let $R$ be a complete, discrete valuation
ring of equal characteristic $p\neq 0$ with algebraically closed residue class
field $k$. By the Cohen structure theorem for complete, local rings, \cite{N},
theorem 31.1 and theorem 31.10, the ring $R$ has a uniquely determined field of
coefficients, which is isomorphic to $k$, and $R$ is isomorphic to $k\llbracket
x\rrbracket$.
\par
Bester constructs the functor $\scrF$ for finite, flat $R$-group schemes $N$,
whose order is a power of $p$, as follows: Choose a smooth, connected, formal
resolution of $N$, i.\,e.\ a family of resolutions
\[
\zeta_i\colon 0\longto N_{R_i}\longto A_i\longto B_i\longto 0
\]
of smooth and connected schemes $A_i$ and $B_i$ over $R_i\defeq R\mod\frakm^i$ such that
$\zeta_i=\zeta_{i+1}\tensor_R R_i$ (cf.\
\cite{bes}, 1, lemma 3.1). We define $\scrF(N)$ to be the pro-sheaf over $\spec k$ associated to
\[
\scrF(N)\defeq(\coker(\pi_1(\alpha_{i\ast}A_i)\to\pi_1(\alpha_{i\ast}B_i)))_{i\in\bbN}.
\]
Here, $\pi_1$ denotes the fundamental group\index{fundamental group} in the sense of proalgebraic
groups\index{proalgebraic group}, \cite{o66}, II.7-4.  Due to the Cohen structure theorem, there
exist canonical morphisms $\alpha_i\colon\spec R_i\to\spec k$. They induce the Weil
restriction\index{Weil restriction}\label{weil_green} functors $\alpha_{i\ast}$, that coincide with
the Greenberg
functor\index{Greenberg functor}, \cite{blr}, p.\ 276. By means of the Greenberg functor, we regard
an $R$-grou as a proalgrabaic group over $k$. Bester shows that $\scrF$ is independent of the resolution chosen
and, indeed, defines a functor.\footnote{Later, we will discuss another approach to Bester's functor $\scrF$.}
\par
The Greenberg functor $(\alpha_{i\ast})_{i\in\bbN}$ has some nice properties: it is exact on smooth
group schemes; more precisely, its first derived functor $\RR^1\alpha_{i\ast} G$ is trivial for
smooth groups $G$ and it respects identity components of smooth group schemes (\cite{bes}, 1, lemma 1.1).
In the following, we will omit the functor $\alpha_{i\ast}$ in the notation and write simply $N$
for the proalgrebraic group $(\alpha_{i\ast}N_{R_i})_{i\in\bbN}$.
\par


%

Bester's pairing is built upon Cartier duality,\index{Cartier duality} $N\times
N^D\to\bbG_m$. Since we are interested in groups $N$ whose order is a power of
$p$, the image of this morphism is contained in some $\mu_{p^n,R}$. In this case,
the Cartier pairing can be equivalently written as
\begin{equation}\label{cartier_p}
N\times N^D\longto\indlim\mu_{p^n,R}\fedeq\mu_{p^\infty},
\end{equation}
for any finite, flat group $N$ whose order is a power of $p$.

The main part of \cite{bes} is devoted to the proof of the following theorem:

\begin{thm}[Bester's Pairing]\label{bester_pairing}\index{Bester's pairing}
\varhspace Let $N$ be a finite, flat $R$-group scheme whose order is a power of
$p$. Then, there exists a perfect pairing
\[
\HH^2_k(R,N)\times \scrF(N^D)\longto\bbQZ.
\]
We will call this pairing \emph{Bester's pairing}.
\qed
\end{thm}

\pagebreak[2]

\begin{remarks}
 \item Let $\tilde G$ be the constant group scheme\index{constant group scheme}
  which is associated to an abstract group $G$. Then, we have $\scrF(\tilde
  G)=G$, \cite{bes}, 1, lemma 3.7 and remark
  3.8.
 \item Let $N$ be any finite, flat $R$-group scheme. If the order of $N$ is
   prime to $p$, then $N$ is \'etale. Since $R$ is strictly henselian, $N$ is
   constant. Thus, $\scrF(N)$ is isomorphic to $N(R)=N(K)$, regarded as a
   constant group scheme, and $\HH^2_k(R,N^D)$ is isomorphic to $\HH^1(K,N^D)$
   by proposition \ref{seq_loc_coh} and lemma \ref{local_coh_lemma}, \enum 1. We
   can define a perfect pairing
   \[
   \HH^2_k(R,N)\times \scrF(N^D)\longto\bbQZ
   \]
   for all finite, flat $R$-group schemes by means of the cup product for the
   prime-to-$p$-part and Bester's pairing for the $p$-part of $N$. Later, we
   will see that Bester's pairing is the well suited extension of the cup
   product to the $p$-part.
\end{remarks}

In the following, we want to study the pairing for tori\index{torus}; more
precisely, we want to study Bester's pairing for the group $\mu_{p^n,R}$ and its
Cartier dual $\bbZ\mod p^n$. Let $\scrG$ denote the N\'eron model of
$\bbG_{m,K}$. It is isomorphic to $\scrG\iso\bigcup_{r\in\bbZ} \pi^r\bbG_{m,R}$
for a uniformising element $\pi\in R$. We will encounter the exact sequence
\begin{equation}\label{n_mult_G}
0\longto\mu_{n,R}\longto\scrG\Longto n n\scrG\longto 0
\end{equation}
for $n\in\bbN_{\geq 1}$. Since the $n$-power map of $\bbG_{m,R}$ is epimorphic
in the flat topology, the subgroup $n\scrG$ of $\scrG$ is given by
$n\scrG=\bigcup_{r\in n\bbZ}\pi^r\bbG_{m,R}$. It is clear that this group scheme
is smooth, too.
\par
We need some cohomological facts:

\begin{lemma}\label{local_coh_lemma}
Let $G$ be a smooth $R$-group scheme.
\begin{enumprop}
  \item $\HH^i(R,G)=0$ for all $i\geq 1$.
  \item $\HH^1_k(R,G)=G(K)\mod G(R)$. In particular we have
  \item $\HH^1_k(R,\bbG_{m,R})=\bbZ$, $\HH^1_k(R,\scrG)=0$.
  \item $\HH^2_k(R,\mu_{n,R})=\HH^1_k(R,n\,\scrG)=\bbZ\mod n$ for all $n\in\bbN$.
  \item $\HH^1_k(R,A^0)=\oldphi_A$.
\end{enumprop}
\end{lemma}

\begin{proof}
\enum 1 is shown in \cite{DixExp}, theorem 11.7. Using this, we can show \enum
2: Let $G$ be a smooth $R$-group scheme. We have the following part of the
sequence of local cohomology:
\[
\HH^0(R,G)\longto\HH^0(K,G)\longto\HH^1_k(R,G)\longto\HH^1(R,G).
\]
Since the latter term is trivial by \enum 1, this implies \enum 2. Since both
$\bbG_{m,R}$ and $\scrG$ are smooth, we get \enum 3 by \enum 2. Let us prove
\enum 4. Consider the exact sequence of \eqref{n_mult_G}. It induces the
following exact cohomology sequence
\[
\HH^1_k(R,\scrG) \longto \HH^1_k(R,{n}\scrG) \longto \HH^2_k(R,\mu_{n,R}) \longto \HH^2_k(R,\scrG).
\]
The group $\HH^1_k(R,\scrG)$ is trivial by \enum 3. Furthermore, the last term
is trivial: The sequence of local cohomology induces the exact sequence
\[
\HH^1(K,\bbG_{m,K})\longto\HH^2_k(R,\scrG)\longto\HH^2(R,\scrG).
\]
We have used $\scrG\tensor_R K=\bbG_{m,K}$. The outer terms of this sequence are
trivial due to Hilbert 90 in the case of $\HH^1(K,\bbG_{m,K})$ and due to \enum
1 in the case $\HH^2(R,\scrG)$. Thus, we have the equality
$\HH^2_k(R,\mu_{n,R})=\HH^1_k(R,n\scrG)$ of \enum 4. Furthermore, since $n\scrG$
is smooth, we can conclude by means of \enum 2 that $\HH^1_k(R,n\scrG)$ is
isomorphic to $(n\scrG)(K)\mod(n\scrG)(R)=K^\ast\mod (n\scrG)(R)$.
\par
Now, consider the diagram
\begin{diagram}
0 & \rTo & \bbG_{m,R} & \rTo & n\scrG & \rTo & i_\ast n\bbZ & \rTo & 0\phantom.  \\
  &      & \dEqual    &      & \dTo  &      & \dTo \\
0 & \rTo & \bbG_{m,R} & \rTo & \scrG & \rTo & i_\ast\bbZ & \rTo & 0.
\end{diagram}
Since $\bbG_{m,R}$ is smooth, this diagram induces a diagram with exact lines of
$R$-valued points. In particular, the cokernel of the second vertical morphism
is $K^\ast\mod (n\scrG)(R)=\HH^1_k(R,n\scrG)$. Obviously, it is isomorphic to
$\bbZ\mod n$.
\par
Finally, \enum 5 is a direct consequence of \enum 2.
\end{proof}

\begin{lemma}\label{coh_f_lemma}
Let $f\colon R_\fl\to R_\et$ denote the canonical morphism of sites induced by
the identity map. Then we have:
\begin{enumprop}
  \item $\RR^rf_\ast G=0$ for a smooth group $G$ and every $r\geq 1$.
  \item $\RR^1f_\ast\mu_{p^n,R}=f_\ast\bbG_{m,R}\mod (f_\ast\bbG_{m,R})^{p^n}$.
    The sheaves $\RR^rf_\ast\mu_{p^n,R}$ are trivial for $r\neq 1$.
  \item $\HH^r_{\et}(R,\RR^1f_\ast\mu_{p^n,R})=\HH^{r+1}_{\fl}(R,\mu_{p^n,R})$.
    These groups are trivial for $r\neq 0$.
  \item
    $\HH^r_{k,\et}(R,\RR^1f_\ast\mu_{p^n,R})=\HH^{r+1}_{k,\fl}(R,\mu_{p^n,R})$.
    For $r=1$, this group is isomorphic to $\bbZ\mod p^n$ and it is trivial for
    $r\neq 1$.
\end{enumprop}
\end{lemma}

\begin{proof}
\enum 1 is \cite{DixExp}, lemma 11.1 with theorem 11.7. \statement 2 is an easy
consequence of the Kummer sequence\index{Kummer sequence}
$0\to\mu_{p^n,R}\to\bbG_{m,R}\stackrel{\tiny p^n}\to\bbG_{m,R}\to 0$ on the flat
site of $\spec R$. It yields the following long exact sequence on the \'etale
site of $\spec R$:
\begin{equation}\label{res_mu_q}
0\longto f_\ast\bbG_{m,R}\Longto{p^n}f_\ast\bbG_{m,R}\longto
\RR^1f_\ast\mu_{p^n,R}\longto\RR^1f_\ast\bbG_{m,R}.
\end{equation}
It is exact on the left, since $f_\ast\mu_{p^n,R}=0$. Now, \enum 1 implies \enum 2.
\par
Statements \enum 3 and \enum 4 follow with the Leray spectral sequences
\begin{align*}
\HH^r_{\et}(R,\RR^sf_\ast\mu_{p^n,R}) & \ssto\HH^{r+s}_{\fl}(R,\mu_{p^n,R}), \\
\HH^r_{k,\et}(R,\RR^sf_\ast\mu_{p^n,R}) & \ssto\HH^{r+s}_{k,\fl}(R,\mu_{p^n,R})
\end{align*}
and with \enum 2: The spectral sequences are degenerate with $E_2^{r,s}=0$ for
$s\neq 1$, hence
$\HH^r_\et(R,\RR^1f_\ast\mu_{p^n,R})=\HH^{r+1}_\fl(R,\mu_{p^n,R})$. Again, the
Kummer sequence induces the long exact cohomology sequence
\[
\HH^r(R,\bbG_{m,R})\Longto{p^n}\HH^r(R,\bbG_{m,R})\longto\HH^{r+1}(R,\mu_{p^n,R})\longto\HH^{r+1}(R,
\bbG_{m,R}).
\]
Since $\HH^r(R,\bbG_{m,R})=0$ for $r\geq 1$, we have $\HH^r(R,\mu_{p^n,R})=0$
for $r\geq 2$.
\par
By the same arguments applied to the sequence of local cohomology, it follows
\enum 4 in the case $r\neq 1$. The case $r=1$ is lemma \ref{local_coh_lemma}.
\end{proof}

In the following, we will often drop the indices ${}_\et$ and ${}_\fl$. However,
it will be unambiguous relative to which site we compute the cohomology groups:
When working relative to the \'etale site, we will use the functor $f_\ast$, to
indicate that $f_\ast N$ refers to the sheaf, defined by the group scheme $N$ on
the small \'etale site. When working on the flat site, we identify the scheme
$N$ with the appropriate sheaf by means of the Yoneda lemma. 
\par
Using these cohomological facts, we want to study Bester's pairing for the
torus $T_K$ and its N\'eron model $T$ of the rigid uniformisation of abelian
varieties. Using lemma \ref{local_coh_lemma}, the group $\HH^1_k(R,T^0)$ is
isomorphic to the group of components of $T$.
\par
Since $T_K$ is split, we can reduce to the case of $T_K\iso\bbG_{m,K}$; thus its
N\'eron model $T$ is isomorphic to $\scrG$ with its identity component
$T^0\iso\bbG_{m,R}$. The exact sequence
\[
0\longto\bbG_{m,R}\Longto\iota p^n\scrG\longto i_\ast p^n\bbZ\longto 0
\]
induces the following exact sequence
\[
\HH^1_k(R,\bbG_{m,R})\Longto\iota\HH^1_k(R,p^n\scrG)\longto\HH^1_k(R,i_\ast p^n\bbZ).
\]
The latter term is trivial: it fits into the se\-quence of local cohomology
(proposition \ref{seq_loc_coh}):
\[
0=\HH^0(K,i_\ast p^n\bbZ)\longto\HH^1_k(R,i_\ast p^n\bbZ)\longto\HH^1(R,i_\ast
p^n\bbZ)=\HH^1(k,p^n\bbZ)=0.
\]
Due to lemma \ref{local_coh_lemma}, there is an epimorphism
\[
\HH^1_k(R,\bbG_{m,R})\Longonto\iota\HH^1_k(R,p^n\scrG)=\HH^2_k(R,\mu_{p^n,R}).
\]
\par
Let $M'=\hom(\bbG_{m,K},\bbG_{m,K})\iso\bbZ$ denote the character group of the
torus $T=\bbG_{m,K}$. Evaluation of characters $\bar
x\mapsto(\phi\mapsto\nu(\phi(x)))$ yields a bijective map
\[
\HH^1_k(R,\bbG_{m,R})=K^\ast\mod R^\ast\Longto{\nu^\ast} M'^\vee.
\]
\par
Finally, the exact sequence $0\to \bbZ\stackrel {p^n}\to\bbZ\to\bbZ\mod p^n\to
0$ induces the connecting morphism of the long $\ext$-sequence
\[
\hom(\bbZ,\bbZ)\Longto\delta\ext^1(\bbZ\mod p^n,\bbZ).
\]
The last group is isomorphic to the Pontrjagin dual of $\bbZ\mod p^n$, hence, we
can write this map as $M'^\vee\to(\bbZ\mod p^n)^\ast$. Since $\scrF(\bbZ\mod
p^n)=\bbZ\mod p^n$, these maps fit into the following diagram:
\begin{equation}\label{bester_eval_diag}
\begin{diagram}[l>=0.4in,midshaft]
\HH^1_k(R,\bbG_{m,R}) & \rOnto^\iota & \HH^2_k(R,\mu_{p^n,R})\phantom. \\
\dTo<{\nu^\ast}             &        & \dTo>{\bp} \\
M'^\vee                & \rTo^\delta & \scrF(\bbZ\mod p^n)^\ast.
\end{diagram}
\end{equation}

To simplify notation, let us write $q\defeq p^n$. The most important result of
this section is the following proposition. It will be crucial to the proof that
Grothendieck's pairing can be described by Bester's pairing.

\begin{prop}\label{bester_eval}\index{Bester's pairing}
Diagram \eqref{bester_eval_diag} commutes, i.\,e.\ Bester's pairing is
compatible with evaluation of characters.
\end{prop}

We will prove this proposition in the setting of derived categories. To do so,
let us recall some basic facts on the derived category. Moreover, we will
translate some previous results into the setting of derived categories. We start
with the following basic ``two of three'' properties for triangulated
categories. They can be seen as a generalisation of the 5-lemma of homological
algebra.

\begin{lemma}
Let
\begin{diagram}
X    & \rTo & Y    & \rTo & Z    & \rTo & X[1] \\
\dTo &      & \dTo &      &      &      & \dTo  \\
X'   & \rTo & Y'   & \rTo & Z'    & \rTo & X'[1]
\end{diagram}
be two distinguished triangles with morphisms $X\to X'$ and $Y\to Y'$ such that
the resulting square commutes. Then there exists a morphism $Z\to Z'$ such that
the resulting diagram is a morphism of triangles. If two of these three
morphisms are isomorphisms, so is the third.
\end{lemma}

\begin{proof}
The first assertion is the axiom [TR3] of \cite{rd}, \S1, the second assertion
is [ibid] proposition 1.1.
\end{proof}

We use the following notations:
\[
  \Gamma\!_{k,R}\defeq \HH^0_k(R,-),\qquad
  \Gamma\!_{R}\defeq \HH^0(R,-),\qquad
  \Gamma\!_{K}\defeq  \HH^0(K,-),
\]
with the corresponding derived functors $\RR\Gamma\!_{k,R}$, $\RR\Gamma\!_R$ and
$\RR\Gamma\!_K$. Since these functors are $\cup$-functors by proposition
\ref{cup_functor_sheaf}, they preserve injective sheaves. This ensures the
existence of various spectral sequences involving those functors.
\par
Let us prove some cohomological facts we need in order to prove proposition
\ref{bester_eval}: There is a ``derived version'' of lemma \ref{coh_f_lemma}:

\begin{lemma}\label{coh_f_der_lemma}\
\begin{enumprop}
  \item $\RR f_\ast\mu_{q,R}  =(\RR^1f_\ast\mu_{q,R})[-1]$.
  \item $\RR\Gamma\!_{k,R,\et}\circ\RR f_\ast\mu_{q,R}=\RR\Gamma\!_{k,R,\fl} \mu_{q,R}
=\HH^2_k(R,\mu_{q,R})[-2]$.
  \item $\RR\Gamma\!_{R,\et}\circ\RR
f_\ast\mu_{q,R}=\RR\Gamma\!_{R,\fl}\mu_{q,R}=\HH^1(R,\mu_{q,R})[-1]$ and analogous over $K$.
  \item $\RR f_\ast\bbZ\mod q=\bbZ\mod q$.
\end{enumprop}
We regard the cohomology objects as a complex concentrated in degrees $1$, $2$,
$1$ and $0$ respectively.
\end{lemma}

\begin{proof}
In the cases of \enum 2 and \enum 3 keep in mind that $f_\ast$ respects
injective sheaves, i.\,e.\ we can use the fundamental equation $\RR\Gamma\circ
\RR f_\ast=\RR(\Gamma\circ f_\ast$), for the appropriate functor $\Gamma_R$ etc.
\par
In all four cases, we have a complex $X$ which has non trivial cohomology in one
level, only (namely $\RR f_\ast\mu_{q,R}$ in level $1$,
$\RR\Gamma\!_{k,R}\mu_{q,R}$ in level 2, etc., see lemma \ref{coh_f_lemma}). By
standard arguments in the derived category, the claimed equality follows.
\end{proof}


There is the following counterpart of the sequence of local cohomology, lemma
\ref{seq_loc_coh}\index{local cohomology}, in the derived world:

\begin{lemma}\label{seq_loc_coh_der}
There is a distinguished triangle
\[
\RR\Gamma\!_{k,R}(X)\longto\RR\Gamma\!_{R}(X)\longto\RR\Gamma\!_{K}(X)\longto\RR\Gamma\!_{k,R}(X)[1]
\]
for every complex $X$.
\end{lemma}

\begin{proof}
As in \cite{rd}, IV, \S 1, ``Motif B'' (p.\ 218).
\end{proof}

If we take cohomology of this triangle, we obtain the exact sequence of local
cohomology, proposition \ref{seq_loc_coh}.

\begin{lemma}\label{coh_F_lemma}\label{pi_1_f_lemma}\
There are canonical isomorphisms
\begin{enumprop}
  \item $\HH^2_k(R,\scrF(\mu_{q,R}))=\bbZ\mod q$.
  \item $\pi_1(\RR^1f_\ast\mu_{q,R})=\scrF(\mu_{q,R})$.
\end{enumprop}
\end{lemma}

\begin{proof}
\enum 1 is shown in the proof of \cite{bes}, 2, lemma 6.2. Statement \enum 2 is
shown for a group of height $1$ in \cite{bes}, 2, lemma 5.13, with lemma 4.7.
For the general case, use the isomorphism
$\scrF(\mu_{q,R})=\pi_1(\RR^1\alpha_\ast\mu_{q,R})$, \cite{bes}, 1, lemma 3.7,
where $\alpha_\ast$ denotes the Greenberg functor, see the introduction to this
section. Using \cite{adt}, remark after theorem III, 10.4, the pro-algebraic
group scheme $\RR^1\alpha_\ast\mu_{q,R}$ can be identified with
$\HH^1(R,\mu_{q,R})$, regarded as a pro-algebraic group scheme. Due to the Leray
spectral sequence, this object can be identified with $\RR^1f_\ast\mu_{q,R}$.
\end{proof}

Finally, we need to evaluate the cup product induced by Cartier duality. Let $S$
be $R$ or $K$. It is easy to see that the Cartier pairing is given by
\index{Cartier duality}
\begin{equation}\label{cartier_map}
\begin{array}{ccl}
\mu_{n,S}\times\bbZ\mod n & \longto & \mu_{n,S}\sub\bbG_{m,S} \\
(\zeta,\bar i) & \longmapsto & \zeta^i.
\end{array}
\end{equation}

The Cartier pairing induces the following cup product:

\begin{lemma}\index{Cartier duality}
Let $N$ be a finite, flat group of order $n$ over $S$. Then, Cartier duality $N\times
N^D\to\mu_{n,S}$
induces  cup products
\begin{align*}
\RR^rf_\ast N\times\RR^sf_\ast N^D & \longto \RR^{r+s}f_\ast\mu_{n,S} \\
\HH^r(S,N)\times\HH^s(S,N^D) & \longto\HH^{r+s}(S,\mu_{n,S})
\end{align*}
for the morphism $f\colon S_\fl\to S_\et$.
\end{lemma}

\begin{proof}
The Cartier pairing $N\times N^D\to\bbG_{m,S}$ induces a morphism $N\tensor
N^D\to\mu_{n,S}$ . By composition with the cup product, theorem \ref{cup_prod},
this induces the desired cup product.
\end{proof}

The description of \eqref{cartier_map} extends to the cup product in the
following sense:

\begin{lemma}\label{cup_description}
Let $S$ be $R$ or $K$.
\begin{enumprop}
  \item $\HH^1(S,\mu_{n,S})=S^\ast\mod(S^\ast)^n$.
  \item The cup product for $r=1$, $s=0$ is given by
    \[
    \begin{array}{ccl}
    \HH^1(S,\mu_{n,S})\times\HH^0(S,\bbZ\mod n) & \longto & \HH^1(S,\mu_{n,S}) \\
    (\bar \zeta,\bar i) & \longmapsto & \overline{\zeta^i}
    \end{array}
  \]
  for an element $\bar\zeta\in\HH^1(S,\mu_n)=S^\ast\mod (S^\ast)^n$
  and $\bar i\in\bbZ\mod n$.
\end{enumprop}
\end{lemma}

\begin{proof}
\enum 1 follows from the Kummer sequence and the triviality of $\HH^1(K,\bbG_m)$
and $\HH^1(R,\bbG_m)$ due to Hilbert 90 and lemma \ref{local_coh_lemma}.
\par
Due to corollary \ref{cech1_cup}, we can describe the cup product in terms of
\Cech\ cohomology: An element of $\cHH^1(S,F)$ can be represented by some
element of $\cHH^1(\frakU,F)$ for some covering \mbox{$\frakU=(U_i\to S)_{i\in
I}$.} Without loss of generality, we can assume that every $U_i$ is connected
such that $(\bbZ\mod n)(U_i)=\bbZ\mod n$. The cup product of \enum 2 on \Cech\
cohomology\index{Cech cohomology} is induced by the Alexander-Whitney formula on
\Cech\ cochains, cf.\
\eqref{alex_whit}\index{Alexander-Whitney formula}
\[
\begin{array}{ccc}\displaystyle
C^1(\frakU,\mu_{n,S}) \times  C^0(\frakU,\bbZ\mod n) & \longto & C^1(\frakU,\mu_{n,S}\tensor\bbZ\mod
n) \\
((\zeta_{j_0,j_1})_{j_0,j_1}, (i_{j_0})_{j_0})         & \longmapsto & \zeta\cup i,
\end{array}
\]
with $(\zeta\cup i)_{j_0,j_1}=\zeta_{j_0,j_1}\tensor i_{j_1}$. The canonical map
$\cHH^1(\frakU,\mu_{n,S}\tensor \bbZ\mod n)\to\cHH^1(\frakU,\mu_{n,S})$ is given
by $\zeta\cup i\mapsto \zeta^i$ with
$(\zeta^i)_{j_0,j_1}=\zeta_{j_0,j_1}^{i_{j_1}}$. The composition of both maps
yields the description of \enum 2.
\end{proof}

Using these results, we can prove the commutativity of \eqref{bester_eval_diag}
in the setting of derived categories:

\begin{proofof}{proposition \ref{bester_eval}}
In a first step, let us elaborate on the construction of Bester's pairing as it
is sketched in \cite{bes}, 2, lemma 6.3. Doing so, we see in which way Bester's
pairing is built upon the cup product for $f\colon R_\fl\to R_\et$. In a second
step, we use this description to evaluate it for $\mu_{q,R}$ and its Cartier
dual $\bbZ\mod q$.
\par
Let $f\colon R_\fl\to R_\et$ be the canonical morphism of sites. In the derived
category, the cup product for $f_\ast$ and $\mu_{q,R}$, $\mu_{q,R}^D=\bbZ\mod q$
is a morphism
\[ \index{cup product}
\RR f_\ast\mu_{q,R}\longto\RR\!\shom(\RR f_\ast\bbZ\mod q,\RR f_\ast\mu_{q,R}).
\]
\par
Bester shows that there is a morphism
\begin{equation}\label{F_morph}
d\colon \rshom(\RR f_\ast\bbZ\mod q,\RR f_\ast\mu_{q,R})\longto \RR\!\shom(\RR\scrF\circ\RR
f_\ast\bbZ\mod q,\RR\scrF\circ\RR f_\ast\mu_{q,R}),
\end{equation}
where $\RR\scrF$ is defined\footnote{Despite of the notation, $\RR\scrF$ is not
the derived functor of $\scrF$. In fact, $\scrF$ defines an exact functor.
Bester chooses this notation because $\pi_1(\RR^1f_\ast N)=\scrF(N)$. Cf.\ lemma
\ref{pi_1_f_lemma}.} to be
$\RR\pi_1[1]\circ\RR\alpha_\ast=\LL\pi_0\circ\RR\alpha_\ast$ for the Greenberg
functor $\alpha_\ast$, cf.\ introduction on page \pageref{weil_green}. Again, we
will omit the functor $\RR a_\ast$ in the notation. By composition with the
derived cup product morphism, this gives rise to a morphism 
\[
\RR f_\ast\mu_{q,R}\longto\RR\!\shom(\RR\scrF\circ\RR f_\ast\bbZ\mod q,\RR\scrF\circ\RR
f_\ast\mu_{q,R}).
\]
\par
As $\RR f_\ast\bbZ\mod q=\bbZ\mod q$ and $\pi_1(\bbZ\mod q)=0$ and $\pi_0(\bbZ\mod q)=\bbZ\mod
q=\scrF(\bbZ\mod q)$, we can conclude that $\RR\scrF\circ\RR f_\ast\bbZ\mod p$ equals $\bbZ\mod q$.
To evaluate $\RR\scrF\circ\RR f_\ast\mu_{q,R}$ we need more preparations: We know the following:
\begin{enumprop}
  \item $\RR f_\ast\mu_{q,R}=\RR^1f_\ast\mu_{q,R}[-1]$, cf.\ Lemma \ref{coh_f_der_lemma},
  \item $\pi_1(\RR^1f_\ast N)=\scrF(N)$,  by lemma \ref{pi_1_f_lemma}, and
  \item $\pi_0(\RR^1f_\ast\mu_{q,R})=0$. This group is trivial since the sheaf
     $\RR^1f_\ast\mu_{q,R}$ fits into an exact sequence
    \[
    0\longto f_\ast\bbG_{m,R}\Longto q f_\ast\bbG_{m,R}\longto\RR^1f_\ast\mu_{q,R}\longto 0,
    \]
   cf.\ \eqref{res_mu_q}. As the (pro-)scheme $f_\ast\bbG_{m,R}$ is connected, so is
   the pro-scheme $\RR^1f_\ast\mu_{q,R}$.
\end{enumprop}
Using arguments similar to those in the proof of lemma \ref{coh_f_der_lemma}, we
can conclude that $\RR\scrF\circ\RR f_\ast\mu_{q,R}=\scrF(\mu_{q,R})$.
\par
Applying $\RR\Gamma\!_{k,R}$, we get
\begin{equation*}\label{bester_derived_0}
\RR\Gamma\!_{k,R}\circ\RR f_\ast\mu_{q,R}\longto\RR\Gamma\!_{k,R}\circ\RR\!\shom(\scrF(\bbZ\mod
q),\scrF(\mu_{q,R})).
\end{equation*}
\par
Again, with lemma \ref{coh_f_der_lemma}, this morphism can be written as
\begin{equation}\label{bester_derived_1}
\HH^2_k(R,\mu_{q,R})[-2] \longto \RR\Gamma\!_{k,R}\circ\RR\!\shom(\scrF(\bbZ\mod
q),\scrF(\mu_{q,R})).
\end{equation}
\par
Next, we want to show that in this very situation the functor
$\RR\Gamma\!_{k,R}$ commutes with $\RR\!\shom$. The exact sequence
$0\to\bbZ\to\bbZ\to\bbZ\mod q\to 0$ induces a triangle
\begin{eqnarray*}
\RR\Gamma\!_{k,R}\circ\rshom(\bbZ,\scrF(\mu_{q,R}))\longto %
\RR\Gamma\!_{k,R}\circ\rshom(\bbZ,\scrF(\mu_{q,R}))\longto \\%
\RR\Gamma\!_{k,R}\circ\rshom(\bbZ\mod q,\scrF(\mu_{q,R}))\longto %
\RR\Gamma\!_{k,R}\circ\rshom(\bbZ,\scrF(\mu_{q,R}))[1].
\end{eqnarray*}
Since the functors $\shom(\bbZ,-)$ (for sheaves) and $\hom(\bbZ,-)$ (for
abstract groups) are isomorphic to the identity, we have
\[
\RR\Gamma\!_{k,R}\circ\RR\!\shom(\bbZ,-)=\rhom(\bbZ,\RR\Gamma\!_{k,R}-).
\]
Therefore, this triangle is isomorphic to the triangle
\begin{eqnarray*}
\rhom(\bbZ,\RR\Gamma\!_{k,R}\scrF(\mu_{q,R}))\longto  %
\rhom(\bbZ,\RR\Gamma\!_{k,R}\scrF(\mu_{q,R}))\longto \\ %
\RR\Gamma\!_{k,R}\circ\rshom(\bbZ\mod q,\scrF(\mu_{q,R}))\longto
\rhom(\bbZ,\RR\Gamma\!_{k,R}\scrF(\mu_{q,R}))[1] %
\end{eqnarray*}
Due to the ``two of three'' properties of triangulated categories, the first
term is isomorphic to $\rhom(\bbZ\mod q,\RR\Gamma\!_{k,R}\scrF(\mu_{q,R}))$.
Since $\scrF(\bbZ\mod q)=\bbZ\mod q$, we can write the cup product morphism
\eqref{bester_derived_1} as
\begin{equation}\label{bester_derived_2}
\beta\colon\HH^2_k(R,\mu_{q,R})[-2]\longto\rhom(\scrF(\bbZ\mod
q),\RR\Gamma\!_{k,R}\scrF(\mu_{q,R})).
\end{equation}
Let us take hypercohomology\index{hypercohomology} of the second term. The
corresponding spectral sequence is
\[
\ext^r(\scrF(\bbZ\mod q),\HH^s_k(R,\scrF(\mu_{q,R})))\ssto\HH^{r+s}(\rhom(\scrF(\bbZ\mod
q),\RR\Gamma\!_{k,R}\scrF(\mu_{q,R}))),
\]
and it induces the edge morphism $E^2\to E_2^{0,2}$:
\[
\eta\colon\HH^2(\rhom(\scrF(\bbZ\mod q),\RR\Gamma\!_{k,R}\scrF(\mu_{q,R})))
\longto\hom(\scrF(\bbZ\mod q),\HH^2_k(R,\scrF(\mu_{q,R}))).
\]
Since $\HH^2_k(R,\scrF(\mu_{q,R}))=\bbZ\mod q$ (lemma \ref{coh_F_lemma}), we can
compose these morphisms to obtain Bester's paring
\[
\bp=\eta\circ\HH^2(\beta)\colon \HH^2_k(R,\mu_{q,R})\longto\hom(\scrF(\bbZ\mod q),\bbZ\mod
q)=\scrF(\bbZ\mod q)^\ast.
\]
\par
Next, we want to show that Bester's pairing is compatible with ordinary cup
products. Therefore, consider the triangle of lemma \ref{seq_loc_coh_der}.
Together with the cup product morphism, it induces the following morphism of
distinguished triangles. The lower right morphism is the morphism $d$ of
\eqref{F_morph}:
\begin{diagram}[width=1.0in,tight]
\RR\Gamma\!_{k,R}\circ\RR f_\ast\mu_{q,R}    & \rTo & \RR\Gamma\!_{k,R}\circ\RR\shom(\bbZ\mod q,\RR
f_\ast\mu_{q,R}) \\
\dTo                                   &      & \dTo \\
\RR\Gamma\!_{R}\circ\RR f_\ast\mu_{q,R}      & \rTo & \RR\Gamma\!_R\circ\RR\shom(\bbZ\mod q,\RR
f_\ast\mu_{q,R}) \\
\dTo                                   &      & \dTo \\
\RR\Gamma\!_{K}\circ\RR f_\ast\mu_{q,R}      & \rTo & \RR\Gamma\!_K\circ\RR\shom(\bbZ\mod q,\RR
f_\ast\mu_{q,R}) \\
\dTo                                   &      & \dTo \\
\RR\Gamma\!_{k,R}\circ\RR f_\ast\mu_{q,R}[1] & \rTo & \RR\Gamma\!_{k,R}\circ\RR\shom(\bbZ\mod q,\RR
f_\ast\mu_{q,R})[1] \\
                                       &      & \dTo>d \\
                                       &      & \RR\Gamma\!_{k,R}\circ\RR\shom(\scrF(\bbZ\mod
q),\RR\scrF\circ\RR f_\ast\mu_{q,R})[1].
\end{diagram}
By the the same ``two of three'' argument as above, the functors $\Gamma\!_R$,
$\Gamma\!_K$ and $\Gamma\!_{k,R}$ commute with $\shom$. Using lemma
\ref{coh_f_der_lemma}, we can write this diagram as
\begin{diagram}[width=1.0in,tight]
\HH^2_k(R,\mu_{q,R})[-2] & \rTo & \rhom(\bbZ\mod q,\HH^2_k(R,\mu_{q,R})[-2]) \\
\dTo                 &      & \dTo \\
\HH^1(R,\mu_{q,R})[-1]   & \rTo & \rhom(\bbZ\mod q,\HH^1(R,\mu_{q,R})[-1]) \\
\dTo                 &      & \dTo \\
\HH^1(K,\mu_{q,R})[-1]   & \rTo & \rhom(\bbZ\mod q,\HH^1(K,\mu_{q,R})[-1]) \\
\dTo                 &      & \dTo \\
\HH^2_k(R,\mu_{q,R})[-1] & \rTo & \rhom(\bbZ\mod q,\HH^2_k(R,\mu_{q,R})[-2])[1] \\
                                    &      & \dTo>d \\
                                    &      & \rhom(\scrF(\bbZ\mod
q),\RR\Gamma\!_{k,R}(\scrF(\mu_{q,R}))[1].
\end{diagram}
Due to corollary \ref{cup_comp} the functors $\RR\Gamma$ and $\RR\Gamma\!_k$
transform the cup product of $f_\ast$ into the cup product for $\HH^0$ or
$\HH^0_k$ on the small \'etale site. Now, we take hypercohomology. Using lemma
\ref{coh_F_lemma} and the edge morphisms $E^1\to E_2^{0,1}$, this induces the following diagram:
\begin{equation*}
\begin{diagram}[width=0.2in,l>=0.5in]
0 \\
\dTo \\
R^\ast\mod (R^\ast)^q & \:\:=\:\: & \HH^1(R,\mu_{q,R})      & \;\;\times\; & \bbZ\mod q  & \rTo^\cup
& \HH^1(R,\mu_{q,R}) \\
\dTo                  &   & \dTo                &            & \dEqual     &           & \dTo \\
K^\ast\mod (K^\ast)^q & = & \HH^1(K,\mu_{q,R})      & \;\;\times     & \bbZ\mod q  & \rTo^\cup &
\HH^1(K,\mu_{q,R}) \\
\dTo>\nu              &   & \dTo                &            & \dEqual     &           & \dTo \\
\bbZ\mod q            & = & \HH^2_k(R,\mu_{q,R})    & \;\;\times     & \bbZ\mod q  & \rTo^\cup &
\HH^2_{k,\fl}(R,\mu_{q,R})  \\
\dTo                  &   & \dEqual             &            & \dEqual     &          & \dTo^d    \\
0                     &   & \HH^2_k(R,\mu_{q,R})    & \times     & \scrF(\bbZ\mod q)  & \rTo^\bp &
\HH^2_{k,\fl}(R,\scrF(\mu_{q,R})) &  =\bbZ\mod q. %
\end{diagram}
\end{equation*}
This diagram shows that Bester's pairing is compatible with the cup product. To
analyse the cup products, we can use the cup product defined on \Cech\
cocycles.\index{Cech cohomology} Lemma \ref{cup_description} shows that it is
given by $(\bar x,\bar i)\mapsto \overline{x^i}$, for $\bar x\in K^\ast\mod
(K^\ast)^{q}$ and $\bar i\in\bbZ\mod q$.
\par
Consequently, Bester's pairing can be described as follows: Lift an element
$\zeta\in\HH^2_k(R,\mu_{q,R})$ to an element $\bar x\in K^\ast\mod
(K^\ast)^{q}$. Then, $\langle \zeta, i\rangle_\bp=\overline{\nu(x^i)}$ in
$\bbZ\mod q$. Consequently, the diagram of pairings
\begin{diagram}[LaTeXeqno]\label{bester_eval_pairing}
\HH^1_k(R,\bbG_{m,R}) & \;\;\;\times & M' & \rTo & \bbZ \\
\dTo              &              & \dTo &    & \dTo \\
\HH^2_k(R,\mu_{q,R})  & \;\;\;\times & \bbZ\mod q & \rTo & \bbZ\mod q & \sub\bbQZ
\end{diagram}
commutes. The upper pairing is given by $(\bar x,n)\mapsto\nu(x^n)$, for $\bar
x\in K^\ast\mod R^\ast$, $n\in M'=\bbZ$. The lower pairing is given by $(\bar
x,\bar n)\mapsto \overline{\nu(x^n)}$ for $\bar x\in K^\ast\mod (K^\ast)^q$,
$\bar n\in\bbZ\mod q$. Thus, diagram \eqref{bester_eval_diag} commutes.
\end{proofof}

\begin{remark}\label{remark_cup_bester}
If we replace the integer $q$ in diagram \eqref{bester_eval_pairing} by an
integer $n$ which is prime to $p$ then we can infer that the cup product
$\HH^2_k(R,\mu_n)=\HH^1(K,\mu_n)\times\HH^0(K,\bbZ\mod n)\to\bbZ\mod n$ is
compatible with the map evaluation of characters, $\nu^\ast$. See proposition
\ref{cup_description}. We will use this observation to combine Bester's pairing
for the $p$-part of a groups and the cup product for the prime-to-$p$-part of a
group to a new pairing that is compatible with the map $\nu^\ast$.
\end{remark}

In \cite{bes}, Bester gives a description of Bester's pairing for groups of
height $1$. We can find the result of proposition \ref{bester_eval} in this
description as follows: Crucial for Bester's description is the exact sequence
\[
0\longto\RR^1f_\ast N\longto\Lie N\tensor\Omega^1_{R/k}\longto \Lie N\tensor\Omega^1_{R/k}\longto 0
\]
(cf.\ \cite{bes}, 2, lemma 4.3). In the case of $N=\mu_{p,R}$, it reads
\[
0\longto f_\ast\bbG_m\mod (f_\ast\bbG_m)^p \mu_{p,R}\Longto\dlog\Omega^1_{R/k}\longto
\Omega^1_{R/k}\longto 0
\]
(cf.\ \cite{adt}, III, example 5.9). The morphism $\dlog\colon x\mapsto
\frac{dx}{x}$ of this sequence is part of the diagram (cf.\ \cite{bes}, bottom
of p. 164)
\begin{diagram}[l>=0.4in,midshaft]
0 & \rTo & \shom(\scrF(\bbZ\mod p),\bbZ\mod p) & \rTo       & \shom(\scrO_k,\scrO_k) \\
  &      & \uTo>\bp                            &            & \uTo>{i_1=\res^\ast} \\
0 & \rTo & \HH^2_k(R,\mu_{p,R})                    & \rTo^\dlog & \HH^1_k(R,\Omega^1_{R/k})
\end{diagram}
where $i_1$ is induced by the residue map\index{residue map} (\cite{bes},
section 2.5). Since $\res\circ\dlog$ equals the valuation on $R$ (\cite{se59},
II, no. 12, proposition 5$'$), we can conclude that Bester's pairing for
$\mu_{p,R}$ and $\bbZ\mod p$ is given by the evaluation map.

\subsection{Bester's Pairing vs.\ Grothendieck's Pairing}

From now on, let $A_K$ be an abelian variety with semistable
reduction\index{semistable reduction} and N\'eron model $A$. Let $M_K$, $E_K$
etc.\ denote the data of rigid uniformisation as in proposition
\ref{rigid_uniform} and \ref{uniform_neron}. 
\par
In the last section, we gave an overview of Bester's pairing for finite, flat
$R$-group schemes. Unfortunately, the kernel of $n$-multiplication $A_n$ of the
N\'eron model of an abelian variety with semistable reduction is known to be
only a quasi-finite, flat group scheme. In \cite{be03}, lemma 14, it is shown
that it is sufficient to consider only the \emph{finite part}\index{finite part
of a group scheme}$A_n^f$ of $A_n$ to obtain a suitable duality result for $A_n$
and $A_n'$. Recall that every quasi-finite $R$-group scheme $N$ is the disjoint
union of open and closed subschemes $N^f$ and $N'$ where $N^f$ is finite and
$N'$ has an empty special fibre (\cite{blr}, 2.3, proposition 4).
\par
Since we are interested in a comparison between Bester's pairing and
Grothendieck's pairing, we assume $n\in\bbN$ to be big enough to kill
$\oldphi_A$ and $\oldphi_{A'}$. To simplify notation, we will just write
$\oldphi$ instead of $i_\ast\oldphi$.
\par
In this situation, we have several important exact sequences:

\begin{prop}\ \label{comp_groups_n}
\begin{enumprop}
  \item There are exact sequences
  \[
  \begin{array}{cl}
  \HH^2_k(R,T_n)\Longto\phi\HH^2_k(R,A_n)\longto\HH^2_k(R,(E'_n)^D) \longto 0 & \quad\text{and}\smallskip \\
  0\longto \scrF(E'_n)\longto \scrF(A'_n)\Longto\psi \scrF(M'\mod nM'). 
  \end{array}
  \]
  where $\scrF(A_n)$ is defined to be $\scrF(A_n^f)$.
 \item The image of $\phi$ is $\oldphi_A$, and
 \item the image of $\psi$ is $\oldphi_{A'}$.
\end{enumprop}
\end{prop}

\begin{proof}
\cite{be03}, lemmas 11, 12 and 13.
\end{proof}

As $n$ does not have to be a power of $p$, the groups $A_n$, $E_n$ etc.\ can
have non-trivial prime-to-$p$-parts. As indicated before, we use the perfect cup
product to extend Bester's pairing from the $p$-part of the groups to the whole
group and call this combined pairing Bester's pairing again. In this spirit, you
have to read the following:
\par
Since $E_n$ is a finite, flat $R$-group scheme, we can apply Bester duality as
in \cite{bes} and we get the following diagram
\begin{equation}\label{bes_def}
\begin{diagram}
\HH^2_k(R,T_n)         & \rTo & \HH^2_k(R,A_n)   & \rTo & \HH^2_k(R,(E'_n)^D) & \rTo & 0\phantom.\\
\dTo>{\!\wr}           &      & \dTo>\Gamma    &      & \dTo>{\!\wr} \\
\scrF(M'\mod nM')^\ast & \rTo & \scrF(A'_n)^\ast & \rTo & \scrF(E'_n)^\ast    & \rTo & 0
\end{diagram}
\end{equation}
for the morphism $\Gamma$ as defined in \cite{be03}, Lemma 14. Essentially, it
is induced by the closed immersion $A_n'^f\into A'_n$ and Bester's pairing for
the finite, flat $R$-group $(A_n'^f)^D$, cf.\ \cite{be03}, proof of  lemma 14.
\par
Since the images of the first horizontal arrows are isomorphic to $\oldphi_A$
and $\oldphi_{A'}^\ast$, the first square of this diagram induces the following
diagram:
\begin{equation}\label{bp_tori_diag}
\begin{diagram}
\HH^2_k(R,T_n)         & \rOnto & \oldphi_A         & \rInto & \HH^2_k(R,A_n)\\
\dTo>{\bp}             &        & \dDashto>{\bp}        &    & \dTo>{\Gamma} \\
\scrF(M'\mod nM')^\ast & \rOnto & \oldphi_{A'}^\ast & \rInto & \scrF(A'_n)^\ast.
\end{diagram}
\end{equation}

We call the morphism between component groups to be induced by Bester's pairing
(as it is the kernel of the right square of \eqref{bes_def}).

\begin{prop}\label{Gamma_bij}
The morphisms $\Gamma$ and $\bp\colon\oldphi_A\to\oldphi_{A'}^\ast$ are
bijective.
\end{prop}

In particular, there exists a perfect pairing between the kernels $A_n$ and
$A'_n$, which we will call Bester's pairing, too. Indeed, on the finite parts of
$A_n$ and $A'_n$ it coincides with the original pairing, defined for finite,
flat $R$-groups.

\begin{proof}
The images of the left horizontal arrows in diagram \ref{bes_def} are
$\oldphi_{A'}$, and $\oldphi_{A}^\ast$ respectively. The first isomorphism in
\eqref{bp_tori_diag} induces a surjection between these groups. Due to duality
reasons, both groups must have the same cardinality, thus the surjection is
indeed bijective. Using the snake lemma (or a diagram chase in \eqref{bes_def})
we see that $\Gamma$ is bijective, too.
\end{proof}

\begin{thm}\label{bester_gp}\index{Grothendieck's pairing}\index{Bester's pairing}
The pairing of component groups, induced by Bester's pairing coincides with
Grothendieck's pairing up to sign.
\end{thm}

\begin{proof}
Let $n$ be big enough with $n\oldphi_A=0=n\oldphi_{A'}$. Due to the remark on
page \pageref{remark_cup_bester}, the cup product, i.\,e.\ Bester's pairing at
its prime-to-$p$-part, is compatible\footnote{Another proof that the cup product
is compatible with Grothendieck's pairing is given in \cite{be03}, 4.1,
proposition 2.} with the map ``evaluation of characters'', $v^\ast$. Thus, we
can argue for Bester's pairing and the cup product simultaneously.
\par
The composition $T\to E\to A$ that induces the morphism
$\oldphi_{T}\to\oldphi_{A}$ has a factorisation $T^0\into nT\to A^0$. Thus,
there is a factorisation
\[
\HH^1_k(R,T^0) \longto \HH^1_k(R,nT) \longto \HH^1_k(R,A^0).
\]
Using lemma \ref{local_coh_lemma}, it corresponds to the first line in the
following diagram (obtained from the right square of diagram
\eqref{bp_tori_diag}):
\begin{diagram}[width=0.5in]
\oldphi_{T} & \rOnto & \HH^2_k(R,T_n) & \rOnto & \oldphi_{A}\phantom.  \\
\dTo>{\nu^\ast}  &      & \dTo>\bp        &      & \dTo<\gp\dTo>\bp \\
M'^\vee       & \rOnto & M'^\vee\mod nM'^\vee & \rOnto & \oldphi_{A'}^\ast.
\end{diagram}
Keep in mind that $M'^\vee\mod nM'^\vee\iso (M'\mod nM')^\ast$, $T^D_n\iso
M'\mod nM'$ and $\scrF(\tilde G)=G$.
\par
In this situation, we know the following:
\begin{enumprop}
  \item The left square commutes by proposition \ref{bester_eval}.
  \item The outer diagram -- consisting of $\nu^\ast$ and $\gp$ -- commutes up
    to sign by corollary \ref{mp_gp2}.
  \item The right square -- consisting of Bester's pairings -- commutes.
\end{enumprop}
Since the morphism $\oldphi_T\to\oldphi_A$ is surjective and since both
compositions with $\bp$ and $\gp$ coincide up to sign, both $\bp$ and $\gp$ must
coincide up to sign.
\end{proof}

%% file: newbester.tex
\subsection{A new view on Bester's pairing}

Let $N$ be a finite, flat group scheme over $R$. Bester's functor $\scrF$ is defined to be the cokernel of
$\pi_1\tilde\alpha_\ast G\to \pi_1\tilde\alpha_\ast H$, where
\[
0\longto N\longto G\longto H\longto 0
\]
is a smooth, connected formal resolution of $N$ and where $\tilde\alpha_\ast$ is the Greenberg
functor.  Since smooth group schemes are acyclic for the Greenberg functor, and $\pi_1$ is exact on
connected proalgebraic group schemes, we want to view this resolution as an acyclic resolution for the functor
$\pi_1\circ\tilde\alpha_\ast$. In this view, we want to regard Bester's functor $\scrF$
as the first derived functor of $\pi_1\circ\tilde\alpha_\ast$.  Unfortunately, this cannot be
made more explicit, since the category of $R$-group schemes (or formal $R$-group schemes) does not
have enough injective objects. However, let us assume we had the machinery of (hyper)cohomology to handle
this composed functor.
\par
In this case, it would be clear that $\scrF$ is independend of the smooth, connected resolution. Moreover, if we assume
that the first right-derived functor of $\pi_1$ is the functor $\pi_0$, the 
Grothendieck spectral sequence would provide the exact sequence of \cite{bes}, I, Lemma 3.7, easily.
\par
In the following we will construct a spectral sequence that serves as a surrogate for the Grothendieck spectral sequence
and  converges to Bester's functor.
\par
Let us write $\RR^0\pi_1$ for the functor $\pi_1$ and $\RR^1\pi_1$ for the functor $\pi_0$ on proalgebraic group schemes.
For $i\neq 0,1$ set $\RR^i\pi_1\defeq 0$

\begin{prop}\label{bester_ss}
Let $G$ be a (formal) $R$-group scheme with a (formal) smooth resolution. Then there exists a spectral sequence
\begin{align*}
\RR^r\pi_1\RR^s\tilde\alpha_\ast G & \ssto\scrF^{r+s} G 
\end{align*}
for a family of objects $\scrF^n(G)$. 
\end{prop}

\begin{remarks}
  \item A priori, this spectral sequence is only functorial in (formal) $R$-groups together with a smooth (formal) resolution. 
  \item More general, we can use (formal) $R$-groups, which have a resolution of group schemes that are $\tilde\alpha_\ast$-
  acyclic.
  \item If $G$ is a smooth $R$-group scheme, we will omit the functor $\tilde\alpha_\ast$ in the notation and write simply 
  $\pi_1 G$ and $\pi_0 G$, respectively.
\end{remarks}

\begin{proof}
Let us begin with the following observation: If $G$ is a pro-algebraic group, let $\bar G$ denote is univeral covering.
(cf.\ \cite{gp}, \S 6.2)
Let us view the canonical map $G^\ast\colon \bar G\to G$ as a complex in degrees $0$ and $1$. Then we have $\HH^i(G^\ast)=\RR^i\pi_1 G$.
Combining this observation with hyper cohomology leads to the result:
\par
Since $G$ has a smooth, hence $\alpha_{i,\ast}$-acyclic resolution $G\to X^\ast$, consider the following double complex $K$.
\begin{diagram}
\tilde\alpha_\ast X^0 & \rTo^{d'} & \tilde\alpha_\ast X^1 & \rTo & \tilde\alpha_\ast X^2 & \rTo & \ldots \\
\uTo<{d''} &     & \uTo &     & \uTo \\
\overline{\tilde\alpha_\ast X^0} & \rTo^{d'} & \overline{\tilde\alpha_\ast  X^1} & \rTo & \overline{\tilde\alpha_\ast  X^2} & \rTo & \ldots
\end{diagram}
The cohomology groups $'\HH^\ast(\tilde\alpha_\ast X^\ast)$ are the objects $\RR^\ast\tilde\alpha_\ast N$. 
Furthermore, since ``universal covering'',  is  an exact functor, it 
commutes with cohomology and we have $'\HH^\ast(\overline{\tilde\alpha_\ast X^\ast})=\overline{\RR^\ast\alpha_{\ast} G}$
\par
 The spectral sequence of hypercohomology
\[
''\HH^r('\HH^s(K))\ssto\HH^{r+s}(\tot K)\fedeq\scrF^{r+s} G
\]
leads to the result. 
\end{proof}

\begin{notation}
For later use let us denote the double complex used in the proof by $K(G)$
\end{notation}

We have seen that the above spectral sequence is functorial in (formal) $R$-groups together with a 
resolution of (formal) smooth $R$-groups. We can weaken this hypothesis:
\par
Since every quotient of a smooth group scheme over $R_i$ is smooth  again, \cite{sga3}, exp. $\rm VI_A$, prop 1.3.1,
the existence of a smooth, formal resolution is equivalent to the existence of a monomorphism $N\into G$ with a
smooth, formal $R$-group $G$. Moreover: 

\pagebreak[2]

\begin{prop}\label{Fi_functor}\
\begin{enumprop}
  \item Let $M\to N$ be a morphism of (formal) $R$-group schemes such that there exist monomorphisms into  smooth (formal) 
  $R$-groups $M\into G$ and $N\into H$ (We do not demand that these monomorphisms are compatible). Then there is 
  a morphism of smooth resolutions of $M$ and $N$. In particular, 
  the above spectral sequence is functorial, once there are monomorphisms into smooth (formal) $R$-groups.
  \item Let 
  \[
  0\longto N'\longto N\longto N''\longto 0
  \]
  be a short exact sequence of (formal) $R$-group schemes such that there exist monomorphisms into  smooth (formal) $R$-groups 
  (Again, we do not demand that these monomorphisms are compatible).
  Then there exist a short exact sequence of smooth (formal) resolutions of $N'$, $N$ and $N''$. In particular, there is
  a long exact sequence
  \[
  0\to\scrF^0(N')\to\scrF^0(N)\to\scrF^0(N'')\to\scrF^1(N')\to\scrF^1(N)\to\scrF^1(N'')\to 0
  \]
\end{enumprop}
\end{prop}

\begin{proof}
The assertion of \enum 1 can be proven as the first part of the proof of \cite{bes}, I, lemma 3.5.
To proof \enum 2 we need to construct in a first step compatible resolutions of $N'$, $N$ and $N''$. This can be done
as in the first part of the proof of ibid. lemma 3.9. Once we have these resolutions, \enum 2 is clear, since 
the functor ``universal covering'' is exact.
\end{proof}

Let us collect some first properties of the functors $\scrF^i$. They can be obtained easily
by the above spectral sequence.

\begin{prop}\label{prop_Fi}
Let $G$ be an $R$-group scheme.
\begin{enumprop}
  \item $\scrF^0(G)=\pi_1\tilde\alpha_\ast G$.
  \item $\scrF^1(G)$ is trivial for connected, smooth groups $G$, more general
  \item $\scrF^0(G)=\pi_1G$ and $\scrF^1(G)=\pi_0 G$ for smooth $R$-group schemes $G$.
\qed
\end{enumprop}
\end{prop}

\begin{prop}\label{newF_besterF}
On the category of (quasi) finite, flat $R$-group schemes, the functors $\scrF^1$  and the functor $\scrF$ of Bester coincide.
\end{prop}

\begin{proof}
Let $0\to N\to G_1\to G_2\to 0$ be a smooth, connected formal resolution of $N$. 
Due to \ref{Fi_functor} it induces the following exact sequence 
\[
\pi_1 G_1\longto\pi_1 G_2\longto\scrF^1(N)\longto \pi_0(G_1)=0.
\]
Hence $\scrF^1(N)$ is the same quotiont as $\scrF(N)$.
\end{proof}

Let $A_K$ be an abelian variety with semistable reduction and let $n\in\bbN$ be big enough to kill $\oldphi$ and $\oldphi'$.
Then there is an exact sequence
\begin{equation}\label{abvar_n}
0\longto A_n\longto A\Longto n A^0\longto 0
\end{equation}
Since $n$-multiplication on $A^0$ is epimorphic, we have the equality $i_\ast\oldphi=A_n\mod A^0_n$. Remember the following 
equalities of lemma \ref{local_coh_lemma}.

\pagebreak[2]

\begin{lemma}
Let $A_K$ be an abelian variety with semistable reduction and let $A$ be its N\'eron model. 
\begin{enumprop}
  \item $\HH^1_k(R,A)=0$ and $\HH^1_k(R,A^0)=\oldphi$.
  \item $\HH^2_k(R,A)=\HH^1(K,A_K)$.
  \item $\HH^i_k(R,(A^0)_n)=\HH^i_k(R,A_n)$ and $\HH^i_k(R,A^0)=\HH^i_k(R,A)$ for $i\geq 1$ .
\qed
\end{enumprop}
\end{lemma}
  
This allows us to reveal the relationship between the parings of Grothendieck, Bester and Shafarevic:

\begin{prop}
Let $A_K$ be an abelian variety with semistable reduction and let $n\in\bbN$ big enough to kill $\oldphi$ and $\oldphi'$.
There is a commutative diagram
\begin{diagram}[width=0in,l>=0.1in]
\HH^1_k(R,A) & \rTo^n & \HH^1_k(R,A^0) & \rTo & \HH^2_k(R,A_n) & \rTo & \HH^2_k(R,A) & \rTo^n & \HH^2_k(R,A) & \rTo & \HH^3_k(R,A_n) & \rTo & 0 \\
\dTo>{\!\wr} &        & \dTo>\gp       &      & \dTo>\bp       &      & \dTo>\SP     &        & \dTo>\SP     &      & \dTo>{\!\wr} \\
(\pi_0 A'^0)^\ast & \rTo^n & (\pi_0A')^\ast &\rTo & \scrF(A_n')^\ast & \rTo & (\pi_1A^0)^\ast & \rTo^n & (\pi_1A)^\ast & \rTo & (\pi_1\tilde\alpha_\ast A_n)^\ast & \rTo & 0  
\end{diagram}
where all vertical morphisms are induced by the perfect pairings of Shafarevic, Bester/Bertapelle and Grothendieck.
\end{prop}

\begin{proof}
The existence of the horizontal sequences is clear.
The squares involving $\bp$ and $\SP$ commute
by \cite{be03}, the square involving $\bp$ and $\gp$ commute by theorem \ref{bester_gp}.
\end{proof}

\begin{remark}
This diagram, of course, induces the sequence of parings as in \eqref{pairing_seq}.
\end{remark}

\noindent{\bf Acknowledgements.} I would like to thank Alessandra Bertapelle and
Siegfried Bosch for many helpful remarks on previous versions of this article
and for valuable dicussions on the arithmetic of abelian varieties.

%% file: bib.tex
\def\MakeUppercase#1{#1}